\documentclass[12pt]{amsart}

\usepackage{graphicx, verbatim,amsmath,amssymb}

\newcommand{\sss}{{\mathcal S}}

\newcommand{\ttt}{{\bf T}}
\newcommand{\Cech}{{\v Cech} }

\newcommand{\ilim}{\varprojlim}
\newcommand{\dlim}{\varinjlim}

\newcommand{\R}{{\mathbb R}}
\newcommand{\Q}{{\mathbb Q}}

\newcommand{\Z}{{\mathbb Z}}

\newtheorem{thm}{Theorem}[section]
\newtheorem*{thm*}{Theorem}
\newtheorem{cor}[thm]{Corollary}

\newcommand{\seabox}{{\framebox{$\searrow$}}}
\newcommand{\neabox}{{\framebox{$\nearrow$}}}
\newcommand{\swabox}{{\framebox{$\swarrow$}}}
\newcommand{\nwabox}{{\framebox{$\nwarrow$}}}

\everymath{\displaystyle}

\theoremstyle{definition}
\newtheorem{ex}{Example}

\begin{document}
\title{Cohomology of Hierarchical Tilings}
\author{Lorenzo Sadun}
\date{\today}

\address{Department of Mathematics\\The University of
  Texas at Austin\\ Austin, TX 78712} \email{sadun@math.utexas.edu}

\maketitle

\setlength{\baselineskip}{.6cm}

Algebraic invariants, such as homotopy groups, homology groups, and
cohomology groups, are used to study topological spaces and maps
between them. The best known of these is homology. In a homology
theory, we associate Abelian groups $C_k(X)$ of {\em chains} to a
topological space $X$, and define a boundary operator $\partial_k:
C_k(X) \to C_{k-1}(X)$ such that $\partial_{k-1}\circ \partial_k =
0$. Chains in the kernel of $\partial_k$ are called closed, and chains
in the image of $\partial_{k+1}$ are called boundaries. The $k$-th
homology of $X$ in this setup is
$H_k(X)=Ker(\partial_k)/Im(\partial_{k+1})$. A continuous map $f: X
\to Y$ induces {\em push-forward} maps $f_*: C_k(X) \to C_k(Y)$.
(Strictly speaking there is one such map for each integer $k$, but
they are all denoted $f_*$). This in turn induces a map (also denoted
$f_*$) from $H_k(X)$ to $H_k(Y)$. Homotopic maps induce the same map
on homology. Homology groups can then help us classify spaces, and the
pushforward $f_*: H_k(X) \to H_k(Y)$ helps classify maps up to
homotopy, and hence the relation between $X$ and $Y$. There are many
different homology theories, including simplicial, singular and
cellular. For CW complexes they all yield isomorphic groups, so we
often get lazy and speak of {\em the} homology of a space $X$ without
specifying the theory.

In a {\em cohomology\/} theory, we associate Abelian groups $C^k(X)$ of
{\em cochains\/} to $X$ and define a {\em coboundary\/} operator $\delta_k
: C^k(X) \to C^{k+1}(X)$ such that $\delta_{k+1}\circ \delta_k
=0$. Given such a setup, the $k$-th cohomology group of $X$ is
$H^k(X)= Ker(\delta_k)/Im(\delta_{k-1})$. A continuous map $f: X \to
Y$ induces {\em pullback\/} maps $f^*:C^k(Y) \to C^k(X)$ and $f^*:
H^k(Y) \to H^k(X)$.

One way to get a cohomology theory is to start with a homology theory
and dualize everything\footnote{This is where the prefix ``co'' for
  objects related to cohomology comes from.}. We can define $C^k(X)$
to be the dual space of $C_k(X)$, and $\delta_k$ to be the transpose
of $\partial_{k+1}$. That is, if $\alpha$ is a $k$-cochain and $c$ is
a $(k+1)$-chain, then
\begin{equation}
(\delta_k \alpha)(c) : = \alpha(\partial_{k+1} c),
\end{equation} 
since the boundary $\partial c$ of $c$ is a $k$-chain\footnote{When
the dimension of a chain or cochain is clear, we often omit the subscript
from $\partial_k$ or $\delta_k$.}. This is
how simplicial, singular, and cellular cohomology are
defined. However, there are also cohomology theories that are defined
intrinsically rather than via homology. In de Rham cohomology, if $X$
is a smooth manifold, then $C^k(X)$ is the set of $k$-forms on $X$,
and $\delta_k$ is the exterior derivative. In \Cech cohomology, the
cochains are defined via open covers of $X$.  Regardless of the setup,
we call elements of $Ker\ \delta_k$ {\em co-closed\/} and elements of
$Im\ \delta_{k-1}$ {\em co-exact\/}, and define
\begin{equation} \label{cohodef} H^k(X) = 
(Ker\ \delta_k)/(Im \ \delta_{k-1}). \end{equation}  

Since the 1990s, \Cech cohomology has been used to study tiling
spaces\footnote{All tilings in this chapter will be assumed to have
  finite local complexity, and in particular to have tiles that meet
  full-edge to full-edge.  Cohomology can also be used to study tiling
  spaces of infinite local complexity, but both the calculations and
  the interpretations are more complicated.}.  This began with work of
Kellendonk \cite{Kel-force-border}, and really took off after the
seminal work of Anderson and Putnam \cite{AP}.  This chapter will
address three essential questions, all of which have generated a host
of papers: (1) What is tiling cohomology?  (2) How do you compute it?
(3) What is it good for?  Most of this chapter is review material, but
the content of Sections 3.1.1 and 3.1.2 is new and is joint work with
John Hunton.

\section{What is tiling cohomology?}

Many algebraic invariants that are used to classify topological spaces
do not work very well with tiling spaces. Tiling spaces (with finite
local complexity) are ``matchbox manifolds''; foliated spaces that
locally look like the product of Euclidean space and a Cantor
set. Tiling spaces have uncountably many path components.  Most of the
standard algebraic invariants are then useless, since they look at
each path component separately, without regard to how the path
components approximate one another. For instance, in singular
homology, $H_0$ of a tiling space is a free group with uncountably
many generators, while all higher homology groups vanish. The
fundamental group and all higher homotopy groups also vanish.

To get around these difficulties, we need to employ less familiar
cohomology theories, especially \Cech cohomology, which is well
adapted to tiling theory. In Subsection \ref{secinverse} we describe
how to view tiling spaces as inverse limits. In Subsection
\ref{secCech} we describe \Cech cohomology and explain how to view the
cohomology of an inverse limit space. In Subsection \ref{secPE} we go
over {\em pattern-equivariant cohomology\/}. This is a theory,
isomorphic to \Cech cohomology, in which the cochains and cocycles can
be viewed as functions on a single tiling.  PV cohomology, described
in subsection \ref{secPV}, is another reformulation of the \Cech
complex, only now the cochains are functions on Cantor sets.  Finally,
in subsection \ref{secQC} we describe quotient cohomology, an analogue
of relative cohomology that is very useful in computations.

\subsection{Inverse limit spaces}\label{secinverse}

Let $\Gamma^0, \Gamma^1, \Gamma^2, \ldots$ be a 
sequence of topological spaces, and for each $n>0$ let 
let $\rho_n: \Gamma^n \to \Gamma^{n-1}$ be a continuous map.
The {\em inverse limit\/} $\ilim (\Gamma^n, \rho_n)$
is a subset of the product space $\prod_n \Gamma_n$. It is the set of all
sequences $(x_0, x_1, x_2, \ldots) \in \prod_n \Gamma^n$ such that for
each $n>0$, $\rho_n(x_n)=x_{n-1}$. The spaces $\Gamma^n$ are called
{\em approximants\/} to the inverse limit, since knowing $x_n \in \Gamma^n$
determines the first $n+1$ terms $(x_0,\ldots,x_n)$ in the sequence, and
thus approximates the entire sequence in the product topology. 

A simple example is the {\em dyadic solenoid\/} $Sol_2$. Each $\Gamma_n$ is the
circle $\R/\Z$, and each $\rho_n$ is the doubling map. A point in
$Sol_2 = \ilim (S^1, \times 2)$ is a point $x_0$ on the unit circle,
together with a choice between two possible preimages $x_1$, another choice 
between possible preimages $x_2$ of $x_1$, another choice of $x_3$, etc. 
Infinitely many discrete
choices make a Cantor set, and $Sol_2$ is a Cantor set bundle over the 
circle. 

There are many descriptions of tiling spaces as inverse limits, and 
we will present a few of the constructions in Section \ref{How-to}. 
If the tilings have finite
local complexity, then the approximants are branched manifolds or
branched orbifolds \cite{AP, BBG, inverse}. Even if the tilings do not have
finite local complexity, it is usually possible to construct
reasonable approximants. The approximants $\Gamma^n$ parametrize the
possible restrictions of a tiling to a ball of radius $r_n$, with
$\lim_{n \to \infty} r_n = \infty$, and the maps $\rho_n$ are obtained
by restricting the tiling to a smaller region. 
A point in
the inverse limit is a set of consistent instructions for tiling bigger
and bigger balls around the origin, which is tantamount to a tiling
of the entire plane. 

\subsection{\Cech cohomology}\label{secCech}

The precise definition of the \Cech cohomology $\check H^*(\Omega)$ of
a topological space $\Omega$ involves the combinatorics of open covers
of $\Omega$, and how the combinatorics change with refinements of the
open covers. The (complicated!) details can be found in an algebraic
topology text \cite{BT, Hatcher, book} and need not concern us here.
What {\em do\/} concern us are some standard properties of \Cech
cohomology.

\begin{thm} If $X$ is a CW complex, then the \Cech cohomology $\check
  H^*(X)$ is naturally isomorphic, as a ring, to the singular
  cohomology $H^*(X)$, and also to the cellular cohomology. If $X$ is
  a manifold, then the \Cech cohomology with real coefficients is
  isomorphic to the de Rham cohomology $H_{dR}^*(X)$.
\end{thm}

Recall that if we have a sequence $G_0, G_1, \ldots$ of groups, and a
collection of homomorphisms $\eta_n^*: G_{n} \to G_{n+1}$, then the
{\em direct limit\/} $\dlim(G_n, \eta_n)$ is the disjoint union of the
$G_n$'s, modulo the relation that $x_n \in G_n$ is identified with
$\eta_n(x_n) \in G_{n+1}$. Every element $x \in \dlim(G_n, \eta_n)$ is
the equivalence class of an element of an approximating group $G_n$;
{\em there are no additional elements ``at infinity''\/}. For instance,
$Z[1/2] := \dlim(\Z, \times 2)$ is isomorphic to the set of {\em
  dyadic rational\/} numbers whose denominators are powers of $2$. The
element $k \in G_n$ is associated with the rational number $k/2^n$,
and $k \in G_n$ equals $2k \in G_{n+1}$ (as it must). The rational
number $5/16$ can be represented as $5 \in G_4$, $10\in G_5$, or $20 \in
G_6$, etc., but has no representative in $G_0$, $G_1$, $G_2$ or $G_3$.
 
\begin{thm} If $\Omega$ is the inverse limit $\ilim (\Gamma^n, \rho_n)$ 
of a sequence
  of spaces $\Gamma^n$ under a sequence of maps $\rho_n: \Gamma^n \to
  \Gamma^{n-1}$, then $\check H^*(\Omega)$ is isomorphic to the direct
  limit $\dlim (\check H^*(\Gamma^n), \rho_{n+1}^*)$.
\end{thm}

In other words, {\em all cohomology theories on a nice space are the
  same, and the \Cech cohomology of an inverse limit is the direct
  limit of the \Cech cohomologies of the approximants.\/}

This is how tiling cohomology is most frequently viewed in
practice. Every element of $\check H^*(\Omega)$ can be represented by a
class in $\check H^k(\Gamma^n)$ on some approximant $\Gamma^n$, and hence by a 
singular or cellular cochain on $\Gamma^n$. Instead
of working with arbitrary open covers of the tiling space itself, we
write everything in terms of the cells that compose the approximants.

As an example, consider the dyadic solenoid.
$H^0(S^1)=H^1(S^1)=\Z$. Since $\rho_n$ wraps the circle
twice around itself, $\rho_n^*$
is the identity on $H^0$ and multiplication by 2 on $H^1$. Thus 
$\check H^0(Sol_2) = \dlim(\Z, \times 1)=\Z$ and $\check H^1(Sol_2)
= \dlim(\Z, \times 2) = \Z[1/2]$.  If we view $S^1$ as consisting of 
one 0-cell and one 1-cell, then for each $m \ge n$, 
the element $2^{-n} \in \check H^1(Sol_2)$
can be represented by a cochain on $\Gamma^m$ that evaluates to
$2^{m-n}$ on the 1-cell.

\subsection{Pattern-Equivariant Cohomology}\label{secPE}

Tiling cohomology can also be understood in terms of the properties of
a single tiling of $\ttt\in \Omega$. This approach, called {\em
  pattern-equivariant (PE) cohomology\/}, was developed by Kellendonk
and Putnam \cite{Kel2,KP} using differential forms, and extended to
integer-valued cohomology in \cite{integer}.

Suppose that $f: \R^d \to \R$ is a smooth function. We say that $f$ is
{\em pattern-equivariant (or PE) with radius $R$\/} 
if the value of $f(x)$ depends
only on what the tiling $\ttt$ looks like in a ball of radius $R$ around
$x$. That is, if $x,y \in \R^d$, and if $\ttt-x$ and $\ttt-y$ agree
exactly on a ball of radius $R$ around the origin, then $f(x)$ must
equal $f(y)$. A function is called {\em strongly PE\/}
if it is PE with
some finite radius $R$.  A function is {\em weakly PE\/}
if it and all of its derivatives are uniform limits of strongly 
PE functions. 

PE forms are defined similarly. Let
$\Lambda_{PE}^k(\ttt)$ denote the $k$-forms on $\R^d$ that are strongly PE with
respect to the tiling $\ttt$.  It is easy to see that the exterior
derivative $d_k$ maps $\Lambda_{PE}^k(\ttt)$ to $\Lambda_{PE}^{k+1}(\ttt)$,
and we define
\begin{equation}\label{PE1}  H^k_{PE}(\ttt,\R) = (Ker\  d_k)/(Im\ d_{k-1}).
\end{equation}

\begin{thm}[\cite{KP}] If $\ttt$ is a tiling with finite local complexity
  with respect to translations, and if $\Omega$ is the continuous hull of
  $\ttt$, then $H^k_{PE}(\ttt,\R)$ is naturally isomorphic to the \Cech
  cohomology of $\Omega$ with real coefficients, denoted $\check
  H^k(\Omega,\R)$.
\end{thm}

To get a PE interpretation of integer-valued cohomology, we use the
fact that a tiling $\ttt$ is itself a decomposition of $\R^d$ into
0-cells (vertices), 1-cells (edges). etc. A PE $k$-cochain $\alpha$ is
a function that assigns an integer to each oriented $k$-cell in a PE
way.  More precisely, there must be a radius $R$ such that, if $c_1$
and $c_2$ are two $k$-cells with centers of mass $x$ and $y$, and if
$T-x$ and $T-y$ agree on a ball of radius $R$ around the origin, then
$\alpha(c_1)=\alpha(c_2)$.  (For integer-valued functions, there is no
distinction between strong and weak pattern-equivariance.) Let
$C^k_{PE}(\ttt)$ denote the set of PE $k$-cochains. 
Instead of the exterior derivative, we consider the cellular
coboundary map $\delta_k$ that maps $C^k_{PE}(\ttt)$ to $C^{k+1}_{PE}(\ttt)$,
and define
\begin{equation} \label{PE2} H^k_{PE}(\ttt) = 
(Ker\ \delta_k)/(Im \ \delta_{k-1}). \end{equation}

\begin{thm}[\cite{integer}] If $\ttt$ is a tiling with finite local
  complexity with respect to translations, and if $\Omega$ is the
  continuous hull of $\ttt$, then $H^k_{PE}(\ttt)$ is naturally isomorphic
  to the \Cech cohomology of $\Omega$ with integer coefficients.
\end{thm}

\begin{proof}[Sketch of proof]
  $\ttt$ induces a map $\pi$ from $\R^d$ to $\Omega$, sending $x \in
  \R^d$ to the tiling $\ttt-x$. Composing with the natural projection
  from $\Omega$ to each approximant $\Gamma^n$, we obtain a sequence
  of maps $\pi_n: \R^d \to \Gamma^n$.  The orbit of $\ttt$ is dense in
  $\Omega$, so these maps are surjective.  Since $\Gamma^n$
  parametrizes the central patch of a tiling, a function on $\R^d$ is
  (strongly) pattern-equivariant if and only if it is the pullback of
  a function on one of the approximants $\Gamma^n$, and the same goes
  for cochains.  Studying $PE$ cochains of arbitrary radius is
  equivalent to studying cochains on $\Gamma^n$ and taking a limit as
  $n \to \infty$. In other words, $H^k_{PE}(\ttt) = \dlim
  H^k(\Gamma^n) \simeq \check H^k(\Omega)$.
\end{proof}

\begin{ex} Let $\ttt$ be a Fibonacci tiling $\ldots babaabaa \ldots$ of
  $\R$ by long (a) and short (b) tiles. Let $i_a$ be a 1-cochain that
  evaluates to 1 on each $a$ tile and 0 on each $b$ tile, and let
  $i_b$ evaluate to 1 on each $b$ and to 0 on each $a$. Since there
  are no 2-cells, $\delta i_a = \delta i_b=0$,
  so $i_a$ and $i_b$ define classes in $H^1_{PE}(\ttt)$. Once we develop
the machinery of Barge-Diamond collaring, we will see that these classes
  correspond to the generators of $\check H^1(\Omega) = \Z^2$.
\end{ex}

\begin{ex} If $\ttt$ is a Thue-Morse tiling $\ldots abbabaabbaababba
  \ldots$, obtained from the substitution $a \to ab$, $b \to ba$, one
  can similarly define indicator 1-cochains $i_a$ and $i_b$ that count
  $a$ and $b$ tiles. However, these cochains are cohomologous. To see
  this, divide the tiling $\ttt$ into 1-supertiles,\footnote{Recall that 
if a substitution tiling is
    non-periodic, then it can be decomposed into supertiles in a
    unique way, and that this decomposition is a local operation. In
    the Thue-Morse tiling, every patch of size $5$ or greater contains
    either the sub-word $aa$ or the sub-word $bb$. The boundaries
    between 1-supertiles sit in the middle of these sub-words, and at
    all points at even distance from these middles.}
with each being either $ab$ or $ba$.  Let $\gamma$ be
  a PE 0-cochain that evaluates to zero on the vertices that mark the
  beginning or end of such a supertile, to 1 on the vertex in the
  middle of an $ab$ supertile, and to $-1$ on the vertex in the middle
  of a $ab$ supertile.  Then $\delta \gamma$ evaluates to 1 on every
  $a$ tile (since the boundary of an $a$ tile is either the middle
  vertex of an $ab$ supertile minus the beginning of that supertile,
  or the end vertex of a $ba$ supertile minus the middle vertex) and
  $-1$ on every $b$ tile, so $\delta \gamma = i_a - i_b$.

  The first \Cech cohomology of the Thue-Morse tiling space is known
  to be $\Z[1/2] \oplus \Z$. The generators can be chosen as follows.
  Let $\alpha_n$ be a 1-cochain that evaluates to 1 on the first tile
  of each $n$-supertile and to 0 on the other $2^n-1$ tiles.  The cochain
$\alpha_n$ basically counts $n$-supertiles. Since there are two $n$-supertiles
in each $(n+1)$-supertile, $\alpha_n$ is cohomologous to $2 \alpha_{n+1}$.
  Let $\beta$ be a 1-cochain that evaluates to 1 on each $a$ tile that
  is followed by a $b$ tile, and to zero on $b$ tiles or on $a$ tiles
  that are followed by $a$ tiles. This is not cohomologous to any
  combination of the $\alpha_n$ tiles since, on average, $\beta$ applied to
  a long interval yields a third of the length of the interval,
  something that no finite linear combination of the $\alpha_n$'s can
  do.  The $\alpha_n$ cochains and $\beta$ generate all of $\check
  H^1$.  In this example, the cochains $i_a$ and $i_b$ are both cohomologous to
  $\alpha_1$.
\end{ex}

\subsection{PV cohomology}\label{secPV}

Another cohomology theory, called PV cohomology after the Pimsner-Voiculscu
exact sequence, was developed by Savinien and Bellissard \cite{SB}. This
theory is based on the structure of the {\em transversal\/} to the tiling
space. Since the $C^*$ algebra associated with a tiling space is constructed
from the transversal and the associated groupoid, this provides a 
more intuitive link
between the cohomology of a tiling space and the K-theory of the $C^*$ algebra.

We associate a distinguished point, called a {\em puncture\/}, to each type
of tile. Usually these are chosen in the interior of the tile, say at the
center of mass, but the precise choice of puncture is unimportant. The
{\em canonical transversal\/} $\Xi$ of a tiling space is the set of tilings for 
which there is a puncture at the origin. This is a Cantor set, and we can
study the ring of continuous integer-valued functions on $\Xi$, denoted
$C(\Xi,\Z)$. 
If $\alpha$ is a $d$-cochain, we define an associated function $f_\alpha$
on $\Xi$ as follows: if $\ttt \in \Xi$, then $f_\alpha(\ttt)$ equals $\alpha$
applied to the tile of $\ttt$ that lies at the origin. This map induces an
isomorphism (as an additive group) between $C(\Xi, \Z)$
and $C_{PE}^d(\ttt)$. 
 
Similarly, we can define punctures for all of the lower-dimensional
faces and edges and vertices of different tiles, with the condition that
if (say) an edge is on the boundary of two tiles, then its puncture viewed
as the boundary of the first tile is the same as its puncture viewed as 
the boundary of the second tile. For $k$ ranging from 0 to $d$, let 
$\Xi^k_\Delta$ be the set of tilings where the origin sits at a puncture of 
an $k$-cell. As with $\Xi = \Xi^d_\Delta$, $C(\Xi^k_\Delta,\Z)$
is isomorphic to $C_{PE}^k(\ttt)$.

In PV cohomology, the group of $k$-cochains is $C(\Xi^k_\Delta,\Z)$ and the 
coboundary maps are built from the geometry of the specific
tiles. After untangling the
definitions, these coboundary maps turn out to be identical to the 
coboundary maps in PE-cohomology. Thus, PV theory and PE theory not only 
have the same cohomologies, but have isomorphic cochain complexes.
For details of this argument, see
\cite{HK}.

\subsection{Quotient cohomology}\label{secQC}

So far we have been discussing the absolute cohomology of each tiling 
space. However, cohomology is also a functor that concerns maps between spaces. 
Inclusions give rise to relative cohomology (see \cite{Hatcher}), while
surjections give rise to a less-known construction called {\em quotient
cohomology}.

Let $f: \Omega_X \to \Omega_Y$ be a factor map of tiling spaces. As long as the 
tilings have finite local complexity with respect to translations, 
the pullback map
$f^*$ is an injection on cochains. (This argument applies both to
\v Cech cochains and to pattern-equivariant cochains.) We then
define the quotient cochain complex $C^k_Q(\Omega_X,\Omega_Y)$ to be 
$C^k(\Omega_X)/f^*(C^k(\Omega_Y))$,
and the quotient cohomology $H^k_Q(\Omega_X,\Omega_Y)$ to be the cohomology of 
this complex. 
The short exact sequence of cochain complexes:
\begin{equation} 0 \to C^k(\Omega_Y) \xrightarrow{f^*} C^k(\Omega_X) \to
C^k_Q(\Omega_X,\Omega_Y) \to 0
\end{equation}
induces a long exact sequence of cohomology groups
\begin{equation}\label{LES} 
\cdots \to \check H^k(\Omega_Y) \xrightarrow{f^*} \check H^k(\Omega_X) \to
H^k_Q(\Omega_X,\Omega_Y) \to  \check H^{k+1}(\Omega_Y) \to \cdots
\end{equation}

As with ordinary relative (co)homology, there is an excision principle:
\begin{thm}\cite{quotient} Let $f: X \to Y$ be a quotient 
map such that $f^*$ is
an injection on cochains. If $Z \subset X$ is an open set 
such that $f$ is
injective on the closure of $Z$, then $H^k_Q(X,Y)$ is isomorphic to
$H^k_Q(X-Z, Y-f(Z))$. 
\end{thm}

For factor maps between tiling spaces, excision cannot be used
directly.  Every orbit is dense, so there are no open sets where $f$
is injective on the closure. However, it is often the case that a
factor map $\Omega_X \to \Omega_Y$ is injective apart from a small set
of tilings.  In such circumstances, one can use homotopy to convert
the tiling spaces into spaces where excision does apply.

\begin{ex} Let $\Omega_X$ be the 1-dimensional tiling space obtained
  from the period-doubling substitution $a \to ab$, $b \to aa$, and
  let $\Omega_Y$ be the dyadic solenoid $Sol_2$ which can be viewed
  formally as coming from a substitution $c \to cc$. (The dyadic
  solenoid is not actually a tiling space, but it has similar
  topological properties, being an inverse limit space, allowing us to
  apply the machinery of quotient cohomology.)  There is a factor map
  $f: \Omega_X \to \Omega_Y$ that identifies two translational orbits
  but is otherwise injective. This map sends a tiling ${\bf T}$ to the
  sequence $(x_0,x_1,\ldots)$, where $x_k$ is the location of the
  endpoints of the $k$-supertiles (mod $2^k$).  In other words,
  $f({\bf T})$ gives the locations of the supertiles of all order in
  ${\bf T}$, but does not indicate which supertiles are of type $a$ or
  type $b$. However, in a period-doubling tiling the $n$-th order
  supertiles are identical except on the very last entry.  Unless the
  tiling ${\bf T}$ consists of two infinite-order supertiles, $f({\bf
    T})$ determines ${\bf T}$. If the tiling ${\bf T}$ does consist of
  two infinite-order supertiles, then there is exactly one other
  tiling ${\bf T'}$, differing from ${\bf T}$ only at a single letter,
  such that $f({\bf T}')=f({\bf T})$.

  In that last instance, we say that ${\bf T}$ and ${\bf T}'$ have a
  {\em 0-dimensional feature}, namely the boundary of an
  infinite-order supertile, and agree away from that feature. Let
  $\Omega_{X_0}=\{{\bf T}, {\bf T}'\}$, and let $\Omega_{Y_0}=f({\bf
    T})$.  The map $f$ is basically a quotient map, identifying the
  orbit of ${\bf T}$ with the orbit of ${\bf T}'$. This identification
  is the {\em suspension} of a map from the 2-point set $\Omega_{X_0}$
  to the 1-point set $\Omega_{Y_0}$.
\end{ex}

The situation of this example is quite common. There are many
situations where a factor map $f: \Omega_X \to \Omega_Y$ between
tiling spaces (or solenoids) is injective except on the translational
orbits of a set $\Omega_{X_0}$ of tilings.  Furthermore,
$\Omega_{X_0}$ has the structure of a $d-\ell$-dimensional tiling
space, admitting an $\R^{d-\ell}$ action and locally being the product
of $\R^{d-\ell}$ and a totally disconnected set.  Defining
$\Omega_{Y_0}$ to be $f(\Omega_{X_0})$, the following theorem relates
the quotient cohomologies of $(\Omega_X, \Omega_Y)$ and
$(\Omega_{X_0}, \Omega_{Y_0})$.

\begin{thm}\cite{quotient} \label{thm-quotient} Let $f: \Omega_X\to \Omega_Y$ 
be a quotient map of tiling spaces
such that $f^*$ is injective on cochains. Suppose that $f$ is injective
aside from the translational orbits of a codimension-$\ell$ set $\Omega_{X_0} 
\subset \Omega_X$ 
of tilings. Let $\Omega_{Y_0} = f(\Omega_{X_0})$. Then $H^k_Q(\Omega_X,\Omega_Y) 
= H^{k-\ell}_Q(\Omega_{X_0}, \Omega_{Y_0})$. 
\end{thm}

In our example, $\ell=1$, $\Omega_{X_0}$ consists of two points,
$\Omega_{Y_0}$ is a single point,
$H^0_Q(\Omega_{X_0},\Omega_{Y_0})=\Z$, and so
$H^1_Q(\Omega_X,\Omega_Y)=\Z$.  Since $\check H^1(\Omega_Y)=\Z[1/2]$,
the long exact sequence (\ref{LES}) shows that $\check H^1(\Omega_X)=
\Z[1/2] \oplus \Z$. This is in fact the first cohomology of the
period-doubling space.

An extension of Theorem \ref{thm-quotient} relates the generators of
$H^{k-\ell}_Q(\Omega_{X_0}, \Omega_{Y_0})$ to the generators of
$H^k_Q(\Omega_X, \Omega_Y)$.  This allows us to construct generators
for $H^k(\Omega_X)$ from generators of $H^k(\Omega_Y)$ and from
generators of $H^*_Q(\Omega_{X_0}, \Omega_{Y_0})$.

\section{How do you compute tiling cohomology?}
\label{How-to}

As with other topological spaces, there is no single ``best'' method 
for computing the cohomology of a tiling space. Different tiling spaces 
are best addressed with different methods. 

Cut-and-project tiling spaces are measurably conjugate to Kronecker flows
on higher-dimensional tori. As topological spaces, they are obtained from
the tori by removing some hyperplanes and gluing them back in multiple times.
Forrest, Hunton and Kellendonk \cite{FHK}, and later Kalugin \cite{Kalugin}
developed ways to compute the cohomology of $\Omega$ from the geometry of the
``window'' used in the cut-and-project scheme. 

Substitution tilings can easily be expressed as inverse limits spaces
in which all the approximants $\Gamma^n$ are homeomorphic to a single
space $\Gamma^0$, and where the substitution $\sigma$ can be viewed as a map
from $\Gamma^0$ to itself.  For these spaces,
computing the cohomology boils down to understanding the cohomology of 
$\Gamma^0$ 
and tracking how the classes evolve under the pullback map $\sigma^*$. 
There are many ways to do this, and
each inverse limit scheme gives rise to a calculational method. 
In this section we develop several such schemes, beginning with the original
ideas of Anderson and
Putnam, and working our way through G\"ahler's construction and the more 
recent ideas of Barge, Diamond, Hunton and Sadun. Variants of
the Anderson-Putnam and Barge-Diamond methods are then applied to
tilings with rotational symmetry and to hierarchical tilings that are
not substitutions (e.g., the ``generalized substitutions'' of \cite{F,
  Marseille}).

Tilings that come from local matching rules are harder to
understand. However, they can sometimes be related to substitution
tilings \cite{Mozes, GS, pinwheel}. When a substitution tiling space
$\Omega_Y$ is the quotient of a local matching rules tiling space 
$\Omega_X$, 
we can study the cohomology of $\Omega_X$ via the cohomology of $\Omega_Y$ 
and the quotient cohomology $H^k_Q(\Omega_X,\Omega_Y)$. 

\subsection{The Anderson-Putnam complex}

Suppose that we have a substitution tiling whose tiles are polygons
that meet full-edge to full-edge. We construct an inverse limit space
whose approximants $\Gamma^n$ describe partial tilings.  Specifically,
a point in $\Gamma^n$ describes where the origin sits within an
$n$-supertile. Since this also determines where the origin sits within
an $(n-1)$-supertile, we have a natural map $\sigma: \Gamma^n \to
\Gamma^{n-1}$ and can consider the inverse limit space $\Omega^0 =
\ilim (\Gamma^n, \sigma)$. 

Since the origin can sit anywhere in a supertile of any type, $\Gamma^n$
consists of one copy of each type of supertile. However, there is an 
ambiguity when the origin sits on the boundary of a supertile. If the origin
sits on the boundary between supertile $A$ and supertile $B$, do we consider
it as part of $A$ or $B$?  The answer is to identify the two edges. 

Specifically, $\Gamma^n$ is obtained by taking the disjoint union of
one copy of each kind of (closed) $n$-supertile, and then applying
the relation that, if somewhere in an admissible tiling an edge $e_1$
of supertile $A$ coincides with an edge $e_2$ of supertile $B$, then $e_1$
and $e_2$ are identified.  

These identifications do not just come in pairs. It may happen that the
right edge of $A$ is identified with the left edges of both $B$ and $C$, and
that the left edge of $C$ is identified with the right edges of both $A$ and 
$D$. In that case, the left edges of $B$ and $C$ and the right edges of $A$
and $D$ would all be identified. The information contained in that point
in $\Gamma^n$ would indicate that the origin either sits at a particular spot
on the right edge of $A$, or at that spot on the right edge of $D$, and also
that it sits at the corresponding spot on the left edge of either $B$ or $C$. 

The set of possible $n$-supertiles looks just like the set
of possible tiles, only scaled up by a factor of $\lambda^n$. As a
result, $\Gamma^n$ is just a scaled-up version of
$\Gamma^0$.
$\Gamma^0$ is called the (uncollared) Anderson-Putnam
complex of the substitution $\sigma$, and is denoted $\Gamma_{AP}$.  
Furthermore, the decomposition
of $n$-supertiles into constituent $(n-1)$-supertiles is combinatorially the
same for all $n$. After rescaling, there is a single map (which we again call 
$\sigma$) from $\Gamma_{AP}$ to itself. This map involves stretching each
tile in $\Gamma_{AP}$ by a factor of $\lambda$, dividing it into tiles via
the substitution rule, and then identifying pieces.  We then define
$\Omega^0 = \ilim(\Gamma_{AP}, \sigma)$. 

\subsubsection{Forcing the border}

{\em Forcing the border\/} was 
defined by Johannes Kellendonk in his
study \cite{Kel-force-border} of the Penrose tiling. 
As we shall see, if a substitution forces the border, then 
$\Omega^0$ is homeomorphic to the
tiling space $\Omega$, allowing for an easy computation of the cohomology
of $\Omega$.  If a substitution doesn't force the border, then
there are a variety of {\em collaring\/} techniques for describing the 
tiling space via a slightly different substitution that does force the border. 
By combining collaring with
the Anderson-Putnam construction, we can compute the cohomology of 
arbitrary substitution tiling spaces.

Suppose we have a non-periodic
substitution tiling space, so that $\sigma: \Omega \to \Omega$ is a
homeomorphism\cite{Mosse,Solomyak}.  
This means that we can decompose each tiling $T$
uniquely into a collection of non-overlapping $1$-supertiles, and by
extension we can decompose $T$ uniquely into non-overlapping
$k$-supertiles for every $k$.  The substitution is said to {\em force
  the border at level $k$\/} if any two $k$-supertiles of the same type
not only have the same decomposition into tiles, but also have the
same pattern of ordinary tiles surrounding them (i.e., the pattern of 
tiles that touch the supertiles at 1 or more points). Moreover, any two 
$n$-supertiles with $n>k$ have the same pattern of $(n-k)$-supertiles
surrounding them. 

\begin{figure}[ht]
\includegraphics[width=3.6truein]{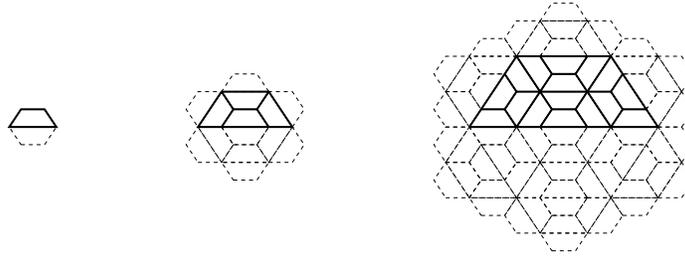}
\caption{In bold face, a half-hex tile, an order-1 supertile, and an order-2
supertile.  In dotted lines, the nearby tiles that these determine.}
\label{half-hex-fig}
\end{figure}

The half-hex substitution is shown in Figure \ref{half-hex-fig}.
The solid lines indicate the tiles within
a supertile, and the dotted lines indicate the neighboring tiles that must
also appear. This substitution forces the border at level 2, since the 
2-supertile is completely surrounded by determined tiles, but does not
force the border at level 1, since some of the tiles
that touch the four vertices of the 1-supertile are undetermined. 
By contrast, the chair tiling does not force the border at all, since 
tiles near the southwest corner of a chair supertile of arbitrary order can
appear in either of the patterns shown in Figure \ref{chair-border-fig}.

\begin{figure}[ht]
\centerline {\raise1.7ex\hbox{\includegraphics[width=1.3truein]{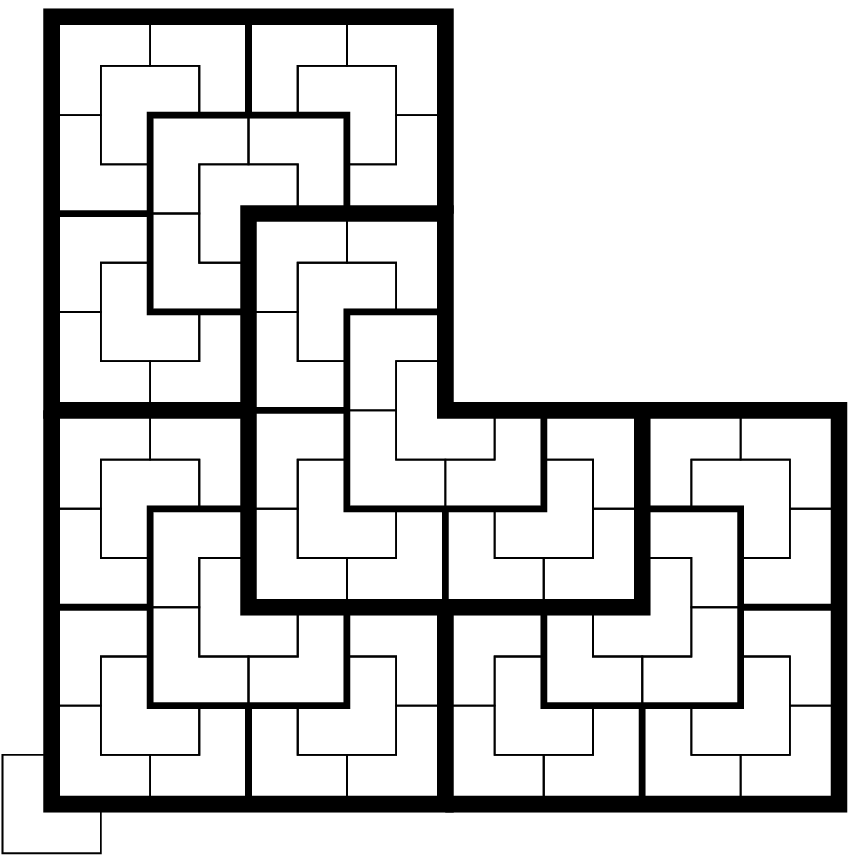}} 
\qquad
\hbox{\includegraphics[width=1.4truein]{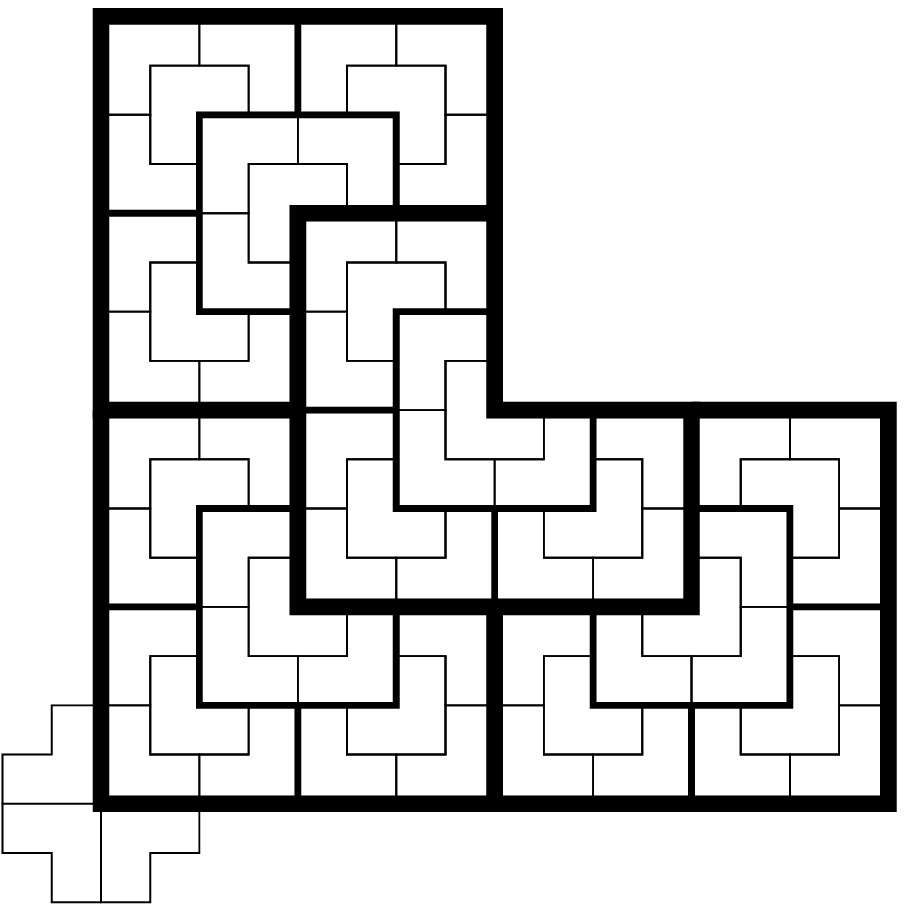}}}
\caption{There are two ways to extend a high-order chair supertile around
the southwest corner.}
\label{chair-border-fig}
\end{figure}

If a substitution forces the border at level $k$, 
then a point in $\Gamma^n$ not only determines where the origin sits in 
a supertile of level $n$, but it determines all of the $(n-k)$-supertiles
surrounding the supertile that contains the origin. If the origin sits on
the boundary between two or more $n$-supertiles, then there is some ambiguity
on the nature of the $n$-supertiles that surround the origin. However, there
is no ambiguity about the $(n-k)$-supertiles that surround the origin. 

The inverse limit $\Omega^0 = \ilim(\Gamma^n_{AP},\sigma)$ is a 
sequence of consistent
instructions for placing higher and higher-order supertiles in a growing
region containing the origin. The union of these regions is all of $\R^d$.
This is tantamount to
\begin{thm}\label{AP1} If $\sigma$ is a substitution that forces the border
and has finite local complexity with respect to translations,
then the corresponding tiling space $\Omega$ is homeomorphic 
to $\Omega^0$.
\end{thm}

\subsubsection{Anderson-Putnam collaring}

If the substitution $\sigma$ does not force the border, then
$\Omega^0$ is typically not homeomorphic to $\Omega$.  There is still
a map $\Omega \to \Omega^0$, whose $n$-th coordinate is a description
of the $n$-supertile containing the origin.  Furthermore, this map is
surjective. However, it is typically not injective. Even if the origin
is not on a boundary, knowing the supertiles to all orders containing
the origin may not describe the entire tiling, since the union of
these supertiles may be a quarter-plane or a half-plane.  If there is
more than one extention of this infinite partial-tiling to the entire
plane, then there is more than one preimage in $\Omega$.

To remedy this, we construct a new substitution using collared tiles.
Take a tiling $\ttt$, and identify tiles that are (a) of the same type
and (b) whose nearest neighbors are all of the same type.  That is,
tiles $t_1$ and $t_2$ are identified if, for some points $x \in t_1$
and $y \in t_2$, the tilings $\ttt-x$ and $\ttt-y$ agree exactly on
the tile containing the origin and on all tiles touching this central
tile. A {\em collared tile\/} is an equivalence class of tiles under this
identification. Note that a collared tile has the {\em same size and shape\/} 
as an
ordinary uncollared tile. The difference is that the label of the collared
tile carries extra information about its surroundings. 

\begin{ex} In the Fibonacci tiling, every $b$ tile is preceded and followed
by an $a$ tile, while an $a$ tile has three possibilities for its neighbors. 
There are thus four collared tiles, which we denote $A_1=(a)a(b)$, 
$A_2=(b)a(a)$, $A_3=(b)a(b)$ and $B=(a)b(a)$, 
where the notation $(x)y(z)$ means a $y$ tile that is 
preceded by  an $x$ and followed by a $z$. Under substitution, 
$A_1 \to (ab)ab(a) = A_3B$, $A_2 \to (a)ab(ab) = A_1B$, $A_3 \to (a)ab(a) 
= A_1B$, and $B \to (ab)a(ab) = A_2$. 
\end{ex}

We can relabel all of our tiles according to their neighbors to obtain a 
new tiling by collared tiles. For instance, in the Fibonacci tiling the
pattern $\ldots babaabaababaa \ldots$ becomes $\ldots BA_3BA_2A_1BA_2A_1B
A_3BA_2A_1 \ldots$. 

\begin{thm}\cite{AP} Rewriting a substitution in terms of collared tiles
always yields a system that forces the border. 
\end{thm}

\begin{proof}[Sketch of proof] A collared tile is a tile together with a 
pattern of nearest neighbors, thereby determining all the tiles in 
at least an $\epsilon$-neighborhood. After substituting $n$ times, 
we obtain an $n$-supertile together with a pattern of neighboring 
$n$-supertiles, thereby determining all the tiles within a distance $\lambda^n
\epsilon$. Pick $n$ big enough that $\lambda^n \epsilon$ is more than twice
the diameter of the largest tile. The $n$-times substituted (collared) tile
then determines its neighboring uncollared tiles 
{\em and the neighbors of these neighbors\/}, and hence
determines its neighboring collared tiles. 
\end{proof}

For instance, in the Fibonacci example, $\sigma^2(A_1) = (aba)aba(ab)
= (BA_2)A_1BA_2(A_1)$, $\sigma^2(A_2)=(ab)aba(aba)=(B)A_3BA_2(A_1B)$,
$\sigma^2(A_3)=(ab)aba(ab)=(B)A_3BA_2(A_1)$ and 
$\sigma^2(B)=(aba)ab(aba) = (BA_2)A_1B(A_3B)$. In each case, substituting
a collared tile twice determines at least two extra tiles on each side
of the 2-supertile, and so determines the collared tile on each side of the
supertile. Combining this theorem with the first Anderson-Putnam construction 
yields the following
\begin{thm}\cite{AP}  Let $\Omega$ be a tiling space derived from a substitution
$\sigma$. Assume that there are only finitely many tile types, 
up to translation, and that the tiles are polygons (or polyhedra) that 
meet full edge to full edge (or full face to full face). Then $\Omega$ is 
homeomorphic to $\ilim(\tilde \Gamma_{AP}, \sigma)$, 
where $\tilde \Gamma_{AP}$ is 
constructed using once-collared tiles. 
\end{thm}

\subsection{G\"ahler's construction}

One can iterate the collaring construction, rewriting an arbitrary tiling
space $\Omega$ in terms of collared tiles, then in terms of 
collared collared tiles (i.e., two
tiles of the same type are identified only if they have the same pattern
of nearest and second-nearest neighbors), and more generally $n$-times 
collared tiles. Let $\Gamma_G^n$ be the Anderson-Putnam complex constructed
from the $n$-times collared tiles. There is a natural quotient 
map $q_n: \Gamma_G^n \to \Gamma_G^{n-1}$ that merely forgets about the $n$-th 
nearest neighbors.  

\begin{thm}Let $\Omega$ be any space of tilings that have finite local 
complexity
with respect to translation. Then $\Omega$ is homeomorphic to the inverse limit
of the approximants $\Gamma_G^n$ under the forgetful maps $q_n$. 
\end{thm}

\begin{proof}[Sketch of proof] (see \cite{Gaehler, inverse}) 
A point $p_n \in \Gamma_G^n$ is either a point
in an $n$-collared tile, or is the identification of several possible 
points on the boundary
of an $n$-collared tile. Either way, at least $n-1$ rings of tiles around
$p_n$ are specified. The point $p_n$ can then be viewed as 
instructions for building a patch {\em around the origin\/}. A 
sequence $p_0, p_1, \ldots$ is then a consistent set of instructions for 
building larger and larger patches around the origin, whose union is 
$\R^d$. Hence $\ilim (\Gamma_G^n, q_n)$ parametrizes tilings in $\Omega$. 
\end{proof}

G\"ahler's construction is extremely useful for theoretical arguments,
as it applies to all tiling spaces, not just to substitution tiling
spaces. For instance, the identification of integer-valued
pattern-equivariant cohomology with \v Cech cohomology \cite{integer}
is based on this construction. Unfortunately, it has not proven
effective in computing cohomology.  $\check H^k(\Omega)$ does equal $\dlim
H^*(\Gamma_G^n)$, but there is no general procedure for computing
$H^*(\Gamma_G^n)$. The number of cells in $\Gamma_G^n$ grows with $n$, and
it is difficult to do computations that apply simultaneously to all
values of $n$.

\subsection{Barge-Diamond collaring}

The Anderson-Putnam and G\"ahler constructions are  
based on collared {\em tiles\/}. The Barge-Diamond
construction \cite{BD, BDHS} is based on collared {\em points}. 

Let $\ttt \in \Omega$ be a non-periodic substitution tiling. Recall that
non-periodicity implies that the substitution $\sigma$ has an inverse
on $\Omega$. Pick a radius $r$ and consider the equivalence relation on
$\R^d$: $x \sim y$ if the tilings $\ttt-x$ and $\ttt-y$ agree out to
distance $r$ around the origin.  Likewise, let $x \sim_n y$ if the
tilings $\sigma^{-n}(\ttt-x)$ and $\sigma^{-n}(\ttt-y)$ agree out to
distance $r$. That is, if $\ttt-x$ and $\ttt-y$ have the same structure of
$n$-supertiles out to distance $\lambda^n r$.  (In particular, 
they also have the
same structure of ordinary tiles out to distance $\lambda^n r$.) Let
$\Gamma_{BD}^n$ be the quotient of $\R^d$ by $\sim_n$. A priori this
would seem to depend on the tiling $\ttt$, but for minimal tiling spaces
all tilings have the same patterns and give rise to identical
approximants.  Since $x\sim_n y $ implies $x \sim_{n-1} y$, there is a
natural quotient map $q_n: \Gamma_{BD}^n \to
\Gamma_{BD}^{n-1}$. Furthermore, the complexes $\Gamma_{BD}^n$ are
all homeomorphic.
Indeed, if $\ttt$ is a self-similar tiling with
$\sigma(\ttt)=\ttt$, then $x \sim y$ if and only if $\lambda^n x \sim_n
\lambda^n y$, so $\Gamma^n_{BD}$ is just an enlarged copy of
a single space $\Gamma_{BD}$ and the quotient maps $q_n$ are all induced
from the substitution $\sigma$.

The radius $r$ is arbitrary, but for many applications it is
convenient to take $r$ extremely small. The complex $\Gamma_{BD}$ is
then a CW complex comprised of pieces of tiles. For instance, suppose
that $\ttt$ is a 1-dimensional tiling.  Points $x$ and $y$ are identified
if either (1) they are in corresponding places in tiles of the same
type, and are farther than $r$ from the nearest vertex, or (2) they
are in corresponding places in tiles of the same type, within distance
$r$ of a vertex, and the tiles on the other side of the vertices are
the same. If the tiles all have length 1, then the equivalence classes
of the first type form 1-cells of length $1-2r$, one for each tile
type. We call these {\em tile cells\/}. The equivalence classes of the
second type form 1-cells of length $2r$, called {\em vertex flaps\/}, one
for each possible transition from one tile to another. For instance,
in the Fibonacci tiling, the possible 2-tile patches are $aa$, $ab$, 
and $ba$, so $\Gamma_{BD}$ consists of two tile cells ($a$ and $b$)
and three vertex flaps, arranged as in Figure \ref{BD-complex-fig}. 
In the Thue-Morse
tiling, all four transitions $\{aa$, $ab$, $ba$, $bb\}$ are possible,
so we have two edge cells and four vertex flaps, also shown in 
Figure \ref{BD-complex-fig}. 

\begin{figure}
\includegraphics[width=3truein]{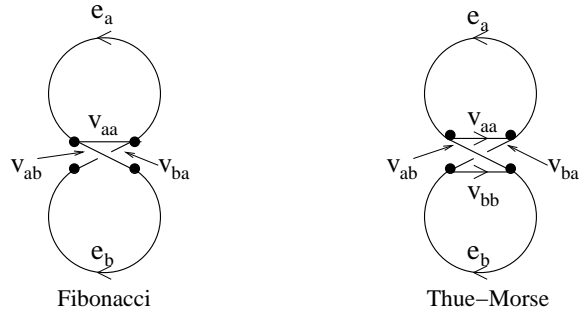}
\caption{Barge-Diamond Complexes for the Fibonacci and Thue-Morse Substitutions}
\label{BD-complex-fig}
\end{figure}

In a 2-dimensional tiling, there are three kinds of 2-cells. {\em Tile
  cells\/} correspond to the interiors of tiles, {\em edge flaps\/}
correspond to points that are within $r$ of an edge, and contain
information about what tile is on the other side of the edge, and {\em
  vertex polygons\/} describe what is happening near a vertex, and have
information about all of the tiles touching the vertex. If the tiles
are unit squares meeting edge-to-edge, 
then the tile cells are $(1-2r)\times(1-2r)$
squares, the edge flaps are $2r \times (1-2r)$ rectangles, and the
vertex polygons are $2r \times 2r$ squares. (Strictly speaking, this
requires using the $L^\infty$ metric on $\R^2$ rather than the
Euclidean metric, to avoid having arcs of circles on the boundaries of
cells.)

\begin{thm}\cite{BD, BDHS} For any positive radius $r$, $\Omega$ is 
homeomorphic  to the inverse limit $\ilim (\Gamma_{BD},\sigma)$.
\end{thm}

\begin{proof} As with the Anderson-Putnam construction, a point in the
  inverse limit is a sequence of instructions for tiling larger and
  larger regions of the plane, insofar as the $n$-th approximant
  determines the structure of a tiling out to distance $\lambda^n r$.
\end{proof}

The complexes $\Gamma_{BD}^n$ are all the same (up to scale), 
so it is relatively easy to compute
$H^*(\Gamma_{BD}^n)=H^*(\Gamma_{BD})$. Unfortunately, the map $\sigma:
\Gamma_{BD} \to \Gamma_{BD}$ is typically not a cellular
map. For instance, for a square tiling $\sigma$ takes a $2r \times 2r$ 
vertex polygon to a $2\lambda r \times 2\lambda r$ square, which is a 
vertex polygon plus a small piece of the adjacent edge flaps and tile cells. 
To do our computations we need to use a map $\tilde \sigma$ that
is cellular and homotopic to $\sigma$. (One way to get such a map
$\tilde \sigma$ is to compose $\sigma$ with a flow that expands tile cells 
slightly at the expense of the edge cells and vertex polygons. The
details are {\em not\/} important.)
The map $\tilde \sigma$
sends vertex polygons to vertex polygons, edge flaps to a union of 
edge flaps and
vertex polygons, and tile cells to a union of all three kinds of cells.  Let
$\tilde \Omega = \ilim(\Gamma_{BD}, \tilde \sigma)$.

\begin{thm} The \v Cech cohomology of $\tilde \Omega$ is isomorphic to the
  \v Cech cohomology of $\Omega$.
\end{thm}

\begin{proof} Since $\sigma$ and $\tilde \sigma$ are homotopic,
  $\tilde \sigma^* = \sigma^*$ as operators on $H^*(\Gamma_{BD})$.
  Then $\check H^*(\tilde \Omega) = \check H^*(\ilim(\Gamma_{BD},\tilde
  \sigma))= \dlim H^*(\Gamma_{BD}, \tilde \sigma^*) = \dlim
  H^*(\Gamma_{BD},\sigma^*) = \check H^*(\ilim(\Gamma_{BD},\sigma)) =
  \check H^*(\Omega)$.
\end{proof}

This theorem does {\em not\/} say that $\tilde \Omega$ and $\Omega$ are 
homeomorphic. In many cases they are not. However, their cohomologies
are the same, so we can always use the inverse limit structure of 
$\tilde \Omega$ to compute the cohomology of $\Omega$. 

\subsubsection{One dimensional results}\cite{BD}

Let $S_0\subset \Gamma_{BD}$ be the sub-complex of vertex flaps, and
let $S_1 = \Gamma_{BD}$. Since $\tilde \sigma$ maps $S_0$ to $S_0$ and
$S_1$ to $S_1$, we can consider the inverse limit space $\sss_i = \ilim
(S_i, \tilde \sigma)$. Since $\sss_0 \subset \sss_1$, we can compute 
$\check H^*(\Omega) = \check H^*(\sss_1)$ by computing 
$\check H^*(\sss_0)$ and
the relative cohomology $\check H^*(\sss_1, \sss_0)$ and then combining them 
with the long exact sequence 
\begin{equation}
0 \to \check H^0(\sss_1, \sss_0) \to \check H^0(\sss_1) \to
\check H^0(\sss_0) \to \check H^1(\sss_1,\sss_0) \to \check H^1(\sss_1) \to 
\check H^1(\sss_0) \to 0
\end{equation}
We examine each of these terms.
$\check H^0(\sss_1, \sss_0)$ is the direct limit (under $\tilde \sigma^*$)
of $H^0(S_1, S_0)$. Since $S_1$ is connected, this is zero. Likewise,
$\check H^0(\sss_1) = \dlim H^0(\sss_1) = \Z$. Since $\tilde \sigma$ maps each
cell of $S_0$ to a single cell, $\tilde \sigma$ merely permutes the cells
of the eventual range $S_0^{ER}$. Thus $\dlim H^*(S_0) = H^*(S_0^{ER})$.
If $S_0^{ER}$ has $k$ connected components and has $\ell$ loops, then
$\check H^0(\sss_0)=\Z^k$ and $\check H^1(\sss_0)=\Z^\ell$. Meanwhile
the quotient space $S_1/S_0$ is a wedge of circles, one for each tile type. 
$H^1(S_1,S_0)=\Z^N$, where $N$ is the number of tile types, and 
$\check H^1(\sss_1, \sss_0) = \dlim (\Z^N, A^T)$, where $A$ is the substitution
matrix. Combining these observations, we have the long exact sequence
\begin{equation}0 \to \Z \to \Z^k \to \dlim (\Z^N , A^T) \to 
\check H^1(\Omega) \to \Z^\ell \to 0.
\end{equation}
Using reduced cohomology, this can be further simplified to 
\begin{equation}0 \to \Z^{k-1} \to \dlim (\Z^N , A^T) \to 
\check H^1(\Omega) \to \Z^\ell \to 0. \label{BD1Dsimple}
\end{equation}

In the Fibonacci tiling, $A= \left ( \begin{smallmatrix} 1 & 1 \cr 1 & 0
\end{smallmatrix} \right )$ and $S_0$ consists of three vertex flaps: 
$aa$, $ab$, and $ba$. These form a contractible set, so $k-1=\ell=0$, 
and $\check H^1(\Omega) = \dlim (\Z^2, A^T) = \Z^2$. In fact, whenever 
$S_0^{ER}$ is contractible, $\tilde H^0(S_0^{ER})$ and 
$H^1(S_0^{ER})$ vanish and $H^1(\Omega)$ is isomorphic to $\dlim (\Z^N, A^T)$. 

We can also describe the Fibonacci tiling using collared
tiles $A_1=(a)a(b)$, $A_2=(b)a(a)$, $A_3=(b)a(b)$, and $B=(a)b(a)$. 
Collaring the Fibonacci tiles and then applying the Barge-Diamond 
construction is overkill, but this example shows the interplay
of the substitution matrix and the cohomology of $S_0^{ER}$.
Our complex $\Gamma_{BD}$ 
has four tile cells and five vertex flaps, namely $A_1B$, $A_2A_1$, $A_3B$,
$BA_2$, and $BA_3$. 
However, $A_1B$ and $A_3B$ are not in $S_0^{ER}$, since all supertiles
start with $A_1$, $A_2$, or $A_3$. $S_0^{ER}$ consists of just the 
flaps $A_2A_1$, $BA_2$ and $BA_3$, yielding 
$k=2$ and $\ell=0$. The substitution matrix is $\left ( 
\begin{smallmatrix} 0&1&1&0 \cr 0&0&0&1 \cr 1&0&0&0 \cr 1&1&1&0 
\end{smallmatrix} \right )$. This matrix has rank 3, with
eigenvalues $(1 \pm \sqrt{5})/2$, $-1$, and 0, and $\dlim(\Z^4,
A^T) = \Z^3$. We then have $0 \to \Z \to \Z^3 \to \check H^1(\Omega) \to 0$. 
After checking that the quotient of $\Z^3$ by $\Z$ is $\Z^2$ (with no 
torsion terms), we again obtain $\check H^1(\Omega) = \Z^2$. 

In the Thue-Morse tiling,
$M= \left ( \begin{smallmatrix} 1 & 1 \cr 1 & 1
\end{smallmatrix} \right )$ and $S_0$ consists of four vertex flaps that form
a loop. Now $\dlim(\Z^2, A^T) = \Z[1/2]$ and $k=\ell=1$, so we have 
$ 0 \to \Z[1/2] \to \check H^1(\Omega) \to \Z \to 0$. Since $\Z$ is free, this 
sequence splits, so $\check H^1(\Omega) = \Z[1/2]\oplus \Z$. 

\subsubsection{Higher dimensions}\cite{BDHS}
In higher dimensions the procedure is similar, but the results cannot
be expressed in a single exact sequence such as (\ref{BD1Dsimple}). In
two dimensions, we consider the complex $S_0$ of vertex polygons,
$S_1$ of vertex polygons and edge flaps, and $S_2 = \Gamma_{BD}$.  We
also consider the inverse limits $\sss_i = \ilim(S_i, \tilde
\sigma)$. As in one dimension, $\tilde \sigma$ maps each vertex
polygon to a single vertex polygon, so $\check H^*(\sss_0) =
H^*(S_0^{ER})$. However, $S_0$ is a 2-dimensional complex, so
computing the cohomology of $S_0^{ER}$ is more than just counting
connected components and loops.

The next step is to consider $\check H^*(\sss_1, \sss_0) = \dlim (\tilde
H^*(S_1/S_0), \tilde \sigma^*)$. This involves only the eventual range
of $S_1$, but is typically a complicated calculation. The quotient
space $S_1/S_0$ breaks into several pieces, one for each direction that an
edge can point. In general, the pieces are not particularly simple, and it 
takes work to understand how $\tilde \sigma^*$ acts on $\tilde
H^*(S_1/S_0)$.
Once $\check H^*(\sss_1, \sss_0)$ is computed, we combine
it with $\check H^*(\sss_0)$ via the long exact sequence
\begin{equation}\label{0to1}
\cdots \to \check H^k(\sss_1, \sss_0) \to \check H^k(\sss_1) \to \check H^k(\sss_0) \to \check H^{k+1}(\sss_1, \sss_0) \to \cdots
\end{equation}
to compute $\check H^*(\sss_1)$.

The relative cohomology $\check H^*(\sss_2, \sss_1)$ is simpler. The
quotient space $S_2/S_1$ is a wedge of spheres, so $\tilde H^0=\tilde
H^1=0$ and $\tilde H^2 = \Z^N$.  $\check H^k(\sss_2, \sss_1)$ equals
$\dlim(\Z^n, A^T)$ when $k=2$, and vanishes when $k=0$ or 1.  The
final stage is combining $\check H^*(\sss_1)$ and $\check H^*(\sss_2,
\sss_1)$ with the long exact sequence
\begin{equation}\label{1to2}
\cdots \to \check H^k(\sss_2, \sss_1) \to \check H^k(\sss_2) \to \check H^k(\sss_1) \to \check H^{k+1}(\sss_2, \sss_1) \to \cdots
\end{equation}

\begin{ex} Consider a tiling of $\R^2$ featuring three square tiles
$A$, $B$, and $C$, and generated by the substitution $\framebox{*} \to 
\begin{matrix} \framebox[16pt]{\mathstrut A}\framebox[16pt]{\mathstrut *} 
\\ \framebox[16pt]{\mathstrut B} \framebox[16pt]{\mathstrut C} 
\end{matrix}$, where ``$*$'' is shorthand for
$A$, $B$ or $C$. This substitution does not force the border, so 
collaring is needed to compute its cohomology. $S_0$ involves many vertex
polygons, but each of these maps to a vertex polygon of the form
$\begin{matrix}\framebox[16pt]{\mathstrut C}\framebox[16pt]{\mathstrut B} \\
\framebox[16pt]{\mathstrut *}\framebox[16pt]{\mathstrut A}\end{matrix}$. 
$S_0^{ER}$ is 
contractible, consisting of three squares glued together at their north
and east edges, so $\check H^0(\sss_0)=\Z$ and $\check H^1(\sss_0)=
\check H^2(\sss_0)=0$. 

$S_1/S_0$ consists of vertical and horizontal edge flaps. The vertical
edge flaps are $B|C$, $C|B$, $A|*$ and $*|A$, but only $C|B$ and $*|A$ 
survive to the eventual range. This portion of $S_1^{ER}/S_0$ retracts 
to the wedge of two circles, and $\tilde \sigma^*$ acts on its first
cohomology by the matrix $\left ( \begin{smallmatrix} 1&1 \cr 1&1
\end{smallmatrix}\right )$, yielding a direct limit of $\Z[1/2]$. The 
horizontal edge flaps are similar, giving another factor of $\Z[1/2]$, 
so $\check H^1(\sss_1, \sss_0)=\Z[1/2]^2$ and $\check H^0(\sss_1,\sss_0) = 
\check H^2(\sss_1, \sss_0)=0$. 

$S_2/S_0$ is a wedge of three spheres, and the only nontrivial cohomology
is $H^2=\Z^3$. This transforms via $A^T=\left ( \begin{smallmatrix}
2&1&1 \cr 1&2&1 \cr 1&1&2 \end{smallmatrix} \right )$, so $\check H^2(\sss_2,
\sss_1) = \dlim(\Z^3, A^T)$ and $\check H^0(\sss_2,\sss_1) = 
\check H^1(\sss_2, \sss_1)=0$. 

We combine these relative cohomologies using the long exact sequences
(\ref{0to1}) and (\ref{1to2}). The first of these yields:
\begin{equation} 0 \to 0 \to \check H^0(\sss_1) \to \Z \to 
\Z[1/2]^2 \to \check H^1(\sss_1) \to 0, 
\end{equation}
so $\check H^0(\sss_1)=\Z$ and $\check H^1(\sss_1) = \Z[1/2]^2$ (and 
$\check H^2(\sss_1)=0$). The second yields:
\begin{equation} 0 \to \check H^1(\sss_2) \to \Z[1/2]^2 \to \dlim(\Z^3,A^T)
\to \check H^2(\sss_2) \to 0.
\end{equation}
All maps commute with $\tilde \sigma^*$. Since the $\Z[1/2]^2$ terms double
with substitution, and since the eigenvalues of $A^T$ are 1, 1, and 4, 
the map from $\Z[1/2]^2$ to $\dlim(\R^3,A^T)$ must be zero. We then have
\begin{equation} \check H^0(\sss_2)=\Z, \qquad \check H^1(\sss_2)=\Z[1/2]^2,
\qquad \check H^2(\sss_2) = \dlim(\Z^3,A^T) = \Z[1/4] \oplus \Z^2.
\end{equation}
(We write $\Z[1/4]$ rather than $\Z[1/2]$ in $\check H^2$ to 
emphasize that this term scales by 4 under substitution.)
This is the same cohomology as the half-hex substitution. In fact, this
tiling space is homeomorphic to the half-hex tiling space. 
\end{ex}

\subsection{Rotations and other symmetries}

A natural question about any pattern is ``what are its symmetries?''
An aperiodic tiling cannot have any translational symmetries, but it
can have rotational or reflectional symmetries. We consider actions of
reflection and rotation (and translation, of course) on the tiling
space $\Omega$, and examine how various quantities transform under
that group action.

\subsubsection{Decomposing by representation}

Rotating a tile and then taking its boundary is the same as taking the
boundary and then rotating. Likewise, rotations commute with {\em
  co}\/boundaries, and in most cases rotations commute with
substitution, so it makes sense to decompose our cochain complexes,
and the cohomology of our tiling space, into representations of
whatever rotation group $G$ acts on our tiling space. By Schurr's
Lemma, neither the coboundary nor substitution can mix different
representations, and our calculations can proceed one representation
at a time.

The trouble with this approach is that representations are vector
spaces, and our cochain complexes take values in $\Z$. We therefore
consider the cohomology of tiling spaces with values in $\R$ rather
than $\Z$. In the process we lose information about torsion and
divisibility, but that's the price we have to pay.

\begin{figure}
\includegraphics[width=2.2truein]{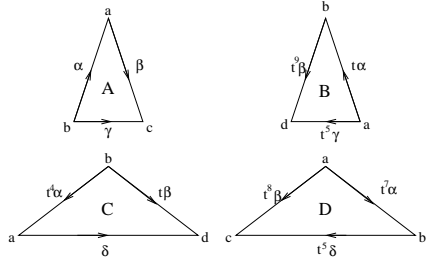}
\caption{Four types of Penrose tiles}
\label{pen-tiles}
\end{figure}

\begin{figure}
\includegraphics[width=3truein]{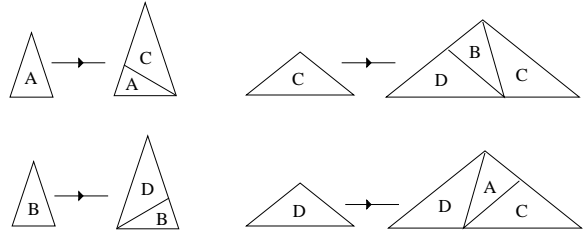}
\caption{The Penrose substitution}
\label{pen-sub}
\end{figure}

For example, the tiles and substitution rules for the Penrose tilings
are shown in Figures \ref{pen-tiles} and \ref{pen-sub}.  There are
four types of tiles, each in 10 orientations, and four types of edges,
each in 10 orientations.  There are only four kinds of vertices
$a,b,c,d$ each of which can sit in the center of a pattern with 5-fold
rotational symmetry.  This means that $a=t^2a$, $b=t^2b$, $c=t^2c$ and
$d=t^2 d$, where $t$ is a rotation by $\pi/5$. In fact, $a=tb$ and
$b=ta$, as can be seen from the fact that the $\alpha$ edge of $A$ 
runs from $b$ to $a$ while the $t\alpha$ edge of $B$ runs from $a$ to $b$.
Likewise, $c=td$ and $d=tc$. This tiling forces the border, so we do
not need to collar.

The group $G=\Z_{10}$ acts on the Anderson-Putnam complex $\Gamma$ by
permuting the tiles, and the eigenvalues of the generator $t$ are the
10th roots of unity.  Each tile type, and each edge type, can be
described by the module $\R[t]/(t^{10}-1)$. The polynomial $t^{10}-1$
factors as $(t-1)(t+1)(t^4+t^3+t^2+t+1)(t^4-t^3+t^2-t+1)$, with the
factors corresponding to the primitive first, second, 5th and 10th
roots, respectively. Each factor also corresponds to a
representation. Since $t^2$ acts trivially on the vertices, only the
representations with $t = \pm 1$ appear in $C_0$.  More specifically,
when working with the Anderson-Putnam complex, our chains complexes
are:
\begin{eqnarray}
C_0(\Gamma)  & = &[\R[t]/(t^2-1)]^2 \cr &=& [\R[t]/(t-1)]^2 \oplus [\R[t]/(t+1)]^2 \cr
C_1(\Gamma) & = & [\R[t]/(t^{10}-1)]^4 \cr &=& \Big [\R[t]/(t-1)  \oplus \R[t]/(t+1)  \oplus
\R[t]/(t^4+t^3+t^2+t+1)  \oplus \R[t]/(t^4-t^3+t^2-t+1)\Big ]^4 \cr 
C_2(\Gamma) & = & [\R[t]/(t^{10}-1)]^4 \cr &=& \Big [\R[t]/(t-1)  \oplus \R[t]/(t+1)  \oplus
\R[t]/(t^4+t^3+t^2+t+1)  \oplus \R[t]/(t^4-t^3+t^2-t+1)\Big ]^4 \cr 
\end{eqnarray}
The complexes $C^k(\Gamma)$ are the dual spaces of $C_k(\Gamma)$.

 The matrices for the boundary maps
$\partial_1: C_1 \to C_0$ and $\partial_2: C_2 \to C_1$ are:
\begin{equation}
\partial_1 = \begin{pmatrix} 1-t & -1 & -t & -1 \cr 0 & 1&1&t \end{pmatrix};
\qquad \partial_2 = \begin{pmatrix}-1&t&t^4&-t^7 \cr -1&t^9&-t&t^8 \cr
1&-t^5 & 0&0 \cr 0&0&1&-t^5 \end{pmatrix}
\end{equation}
in the representations $t = \pm 1$. In the other representations $\partial_2$
is the same, but $\partial_1$ is identically zero (since $C_0=0$). The
coboundary maps $\delta_0$ and $\delta_1$ are the transposes of 
$\partial_1$ and $\partial_2$, only with $t$ replaced by
$t^{-1}$. 

In the $t=1$ representation, $\delta_1$ has rank 2 and $\delta_0$ has
rank 1, and we get $H^0=H^1=\R$ and $H^2=\R^2$. These are the elements of 
cohomology that are invariant under rotation. We say that this portion of the cohomology
rotates like a scalar. 

In the $t=-1$ representation, $\delta_0$ and $\delta_1$ are each rank 2,
and we get $H^2=\R^2$ and $H^1=H^0=0$. This portion of the cohomology
rotates like a {\em pseudoscalar}, flipping sign with every 36 degree rotation. 

In the representation with $t^5=1$ but $t \ne 1$ (that is, with $t^4+t^3+t^2+t
+1=0$), $\delta_1$ is a rank-4 isomorphism, so all cohomologies vanish.
In the representations with $t^5=-1$ (but $t \ne -1$), $\delta_1$ has rank
3, so $H^1=H^2=\R[t]/(t^4-t^3+t^2-t+1)$.  This portion of the cohomology rotates 
like a {\em vector}, flipping sign after a 180 degree rotation.

Substitution acts on 2-cells by the matrix 
$\left( \begin{smallmatrix} t^7&0&0&t^4 \cr 0 & t^3 & t^6 & 0 \cr t^3 &0&t^4&1
\cr 0 & t^7&1&t^6 \end{smallmatrix} \right )$ and on 1-cells by
$\left( \begin{smallmatrix} 0&0&0&t^8 \cr t^4 &0&-t^7 & 0 \cr -t^7 & 0 &0&0
\cr 0 & -t^3 & 0 & -t^3 \end{smallmatrix} \right )$. Both of these matrices
are invertible for all representations (in fact, both have determinant 1), 
so $\check H^k(\Omega) = H^k(\Gamma)$ for $k=0,1,2$. In summary:
\begin{eqnarray}
\check H^0(\Omega) = H^0(\Gamma) &=& \R[t]/(t-1) \cr 
\check H^1(\Omega)= H^1(\Gamma) & = & \R[t]/(t-1) \oplus \R[t]/(t^4-t^3+t^2-t+1) \cr 
\check H^2(\Omega)= H^2(\Gamma) & = & (\R[t]/(t-1))^2 \oplus (\R[t]/(t+1))^2 \oplus \R[t]/(t^4-t^3+t^2-t+1)
\end{eqnarray}

The upshot is that $\check H^0(\Omega) = \R$ and is rotationally
invariant, which is no surprise, since the generator is the constant
function. $\check H^1(\Omega)= \R^5$, of which 4 dimensions rotate
like vectors, with $t^5=-1$, and one is rotationally
invariant. $\check H^2(\Omega)=\R^8$, consisting of a rotationally
invariant $\R^2$, a piece $\R^2$ that rotates like a pseudoscalar, and
a piece $\R^4$ that rotates like a vector.

\subsubsection{Three tiling spaces}

For 2-dimensional substitution like the Penrose tiling and the chair
tiling, there are actually three tiling spaces to be considered. We
have been considering the space $\Omega$ that is the (translational)
orbit closure of a single tiling. This would be, for instance, the set
of all chair tilings where the edges are parallel to the coordinate
axes. We can also consider a larger space $\Omega_{rot}$ of all
rotations of tilings in $\Omega$. Finally we can consider the quotient
space $\Omega_0$ of tilings modulo rotations. $\Omega_0$ can either be
viewed as $\Omega_{rot}/S^1$ or as the quotient of $\Omega$ by the
discrete group of rotations that acts on $\Omega$. For the Penrose
space, we would have $\Omega_0 = \Omega/\Z_{10}$, while for the chair
tiling we would have $\Omega_0 = \Omega/\Z_4$.

The cohomologies of the three spaces are related as follows \cite{ORS, BDHS}:

\begin{thm}\label{XxS1} 
Working with real or complex coefficients, the cohomology of 
$\Omega_0$ is isormorphic to the rotationally invariant part of the cohomology
of $\Omega$. The cohomology of $\Omega_{rot}$ is isomorphic to the cohomology of 
$\Omega_0 \times S^1$. 
\end{thm}

The upshot of this theorem is that $\Omega$ is the space with the
richest cohomology, while $\Omega_{rot}$ and $\Omega_0$ have less
cohomological structure. This is because all rotations on
$\Omega_{rot}$ are homotopic to the trivial rotation, and so act
trivially on $\check H^*(\Omega_{rot})$. Thus, only the rotationally
invariant parts of the cohomology of $\Omega$ can manifest themselves
in the cohomology of $\Omega_{rot}$.

In general, $\Omega_{rot}$ is not homeomorphic to $\Omega_0 \times
S^1$, since the action of $S^1$ on $\Omega_{rot}$ is typically not
free. There are some tilings in $\Omega_{rot}$ that have discrete
$k$-fold rotational symmetry. For these tilings, rotation by $2 \pi/k$
brings us back to the same tiling. $\Omega_{rot}$ has the structure of
a circle bundle over $\Omega_0$ with some exceptional fibers
corresponding to these symmetric tilings.  (Seifert fibered
3-manifolds have a very similar structure.) When working with integer
coefficients, these exceptional fibers can give rise to torsion in
$\check H^2(\Omega_{rot})$.\cite{BDHS}

These relations can also be used to compute the cohomology of the
pinwheel tiling space. When there are tiles that point in all
directions, the only two well-defined spaces are $\Omega_{rot}$ and
$\Omega_0$. For the pinwheel, the cohomology of $\Omega_0$ can be
computed with Barge-Diamond collaring, with the result
that $\check H^0(\Omega_0)=\Z$, $\check H^1(\Omega_0)= \Z$ and $\check
H^2(\Omega_0)= \Z[1/5] \oplus \Z[1/3] \oplus \Z^5 \oplus \Z_2$. This
then determines the real cohomology of $\Omega_0$, and, by Theorem
\ref{XxS1}, the real cohomology of $\Omega_{rot}$. To compute the
integer cohomology of $\Omega_{rot}$, we have to consider the
exceptional fibers. There are 6 pinwheel tilings with 2-fold
rotational symmetry, as shown in Figure \ref{pin-sym2} below; 
these give rise to a $\Z_2^5$ term in a spectral
sequence
$$
\hskip -2in
\begin{picture}(130, 120) 
\setlength{\unitlength}{2pt}
\put(0,0){\vector(0,3){55}} 
\put(0,0){\vector(3,0){130}} 
\put(20,0){\line(0,3){40}} 
\put(45,0){\line(0,3){40}} 
\put(120,0){\line(0,3){40}} 
\put(0,20){\line(3,0){120}} 
\put(0,40){\line(3,0){120}} 
\put(8,28){$\Z$}
\put(23,28){$\Z \oplus \Z_2^5$}
\put(48,28){$\Z[1/5]\oplus\Z[1/3]^2\oplus \Z^5 \oplus \Z_2$}
\put(8,8){$\Z$}
\put(23,8){$\Z$}
\put(48,8){$\Z[1/5]\oplus\Z[1/3]^2\oplus \Z^5 \oplus \Z_2$}
\end{picture}
$$
that computes the cohomology of $\Omega_{rot}$. Furthermore, the $d_2$
map in the spectral sequence involves the torsion elements in a
non-trivial way.  The end result is that $\check H^1(\Omega_{rot})=
\Z^2$, $\check H^2(\Omega_{rot})= \Z[1/5] \oplus \Z[1/3]^2 \oplus \Z^6
\oplus \Z_2^5$ and $\check H^3(\Omega_{rot})= \Z[1/5] \oplus \Z[1/3]^2
\oplus \Z^5 \oplus \Z_2$. For details of this calculation, see
\cite{BDHS} or \cite{book}.

\section{What is cohomology good for?}

\subsection{Distinguishing spaces} The most obvious use of topological
invariants such as \v Cech cohomology is to distinguish spaces. If
tiling spaces $\Omega$ and $\Omega'$ have cohomologies that are not
isomorphic (as rings), then $\Omega$ and $\Omega'$ cannot be
homeomorphic. If a group $G$ of isometries of $\R^d$ (such as
$\Z_{10}$ or $\Z_4$) acts on $\Omega$ and $\Omega'$, then we can
decompose each cohomology group into representations of $G$.  For each
irreducible representation $\rho$ of $G$, let $\check
H^k_\rho(\Omega)$ be the part of $\check H^k(\Omega)$ that transforms
under $\rho$.  If $\Omega$ and $\Omega'$ are homeomorphic via a map
that intertwines the action of $G$, then for each representation
$\rho$ we must have $\check H^k_\rho(\Omega) = \check
H^k_\rho(\Omega')$.  In particular, if a tiling space $\Omega'$ is
related to the Penrose tiling $\Omega$ by a $\Z_{10}$-equivariant
homeomorphism, then not only must $\check H^1(\Omega',\R)$ equal
$\R^5$, but $\check H^1(\Omega',\R)$ must consist of a 1-dimensional
rotationally invariant piece and a 4-dimensional piece that rotates
like a vector.

For each subgroup $H <G$, we can also consider the topology of the set
$\Omega_{H}$ of tilings in $\Omega$ that are fixed by $H$.  If
$\Omega$ and $\Omega'$ are tiling spaces with the same rotation group
$G$, and if there exists an isomorphism that commutes with the action
of $G$, then $\Omega_{H}$ and $\Omega'_{H}$ must be homeomorphic. In
this sense, the structure of $\Omega_{H}$ is a topological invariant
of the tiling space $\Omega$.  If $H_1 < H_2$, then $\Omega_{H_2}
\subset \Omega_{H_1}$. The way that these different spaces nest within
one another is also manifestly invariant.

If $d=2$ and $G$ is a subgroup of $SO(2)$, then $\Omega_{H}$ is not especially
interesting. If $N$ is the normalizer of $H$ in $G$, then $N$ acts on
$\Omega_{H}$, and there are typically only finitely many orbits. Understanding
$\Omega_{H}$ boils down to counting these orbits and identifying how much 
symmetry a point in each orbit has. 

Things get more interesting if $d>2$, or if $H$ involves reflections. In that
case, there may be a subspace $V \subset \R^d$ whose vectors are fixed by 
the action of $H$. $\Omega_{H}$ is invariant under translation by elements
of $V$, and can often be realized as a space of tilings of $V$, or 
as a disjoint union of several such lower-dimensional tiling spaces. In such
cases, the \Cech cohomology of $\Omega_{H}$ yields an interesting invariant. 
 
We present two worked examples. We first
consider the chair tiling of the plane, with $G=D_8=O(2,\Z)$, the
group generated by rotation by 90 degrees and by reflection about the
$x$ axis. We compute the structure of $\Omega_{H}$ for every nontrivial
subgroup $H<G$. We then consider the pinwheel tilings,
with $G=O(2)$. 

\subsubsection{The chair tiling}

We work with the ``arrow'' version of the chair tiling. This is a
2-dimensional substitution tiling in which the tiles are all unit 
squares that meet
full-edge to full-edge. Each square is decorated with an arrow
pointing northeast, southeast, northwest or southwest. Rotation and
reflection act naturally on arrows, so a counterclockwise rotation by
90 degrees would send a northeast arrow to a northwest, a northwest to
a southwest, a southwest to a southeast, and a southeast to a northeast. 
Likewise, reflection about the $x$ axis interchanges northeast and 
southeast arrows, and interchanges northwest and southwest arrows. 
The substitution on northeast arrows is
\begin{equation*}\neabox \longrightarrow 
\begin{matrix} \seabox\neabox \\ \neabox\nwabox \end{matrix}, 
\end{equation*}
and the substitution on all other arrow tiles is obtained by rotating
or reflecting this picture. 
 
There are nine nontrivial subgroups of $G=D_8$. These include $H_1=G$ itself,
the rotation groups $H_2=\Z_4$ and $H_3=\Z_2$, the dihedral group 
$H_4=D_4$ generated by
reflections about the $x$ and $y$ axes, and the 2-element groups $H_5$ 
generated by reflection
about the $x$ axis, $H_6$ generated by reflection about the 
$y$ axis, $H_7$ generated by reflection about the line $y=x$, and 
$H_8$ generated by reflection about the line $y=-x$.  Finally, there is the
dihedral group $H_9$ generated by $H_7$ and $H_8$. 

There is only one tiling that is invariant under all of $G$, namely the
fixed point of the substitution whose central patch involves four arrows
pointing out from the origin: 
$ \begin{matrix} \nwabox\neabox \\ \swabox \seabox \end{matrix}$.
This is also the only tiling that is invariant 
under $H_2$ or $H_3$ or $H_4$ or $H_9$.  
$H_5$ and $H_6$ are conjugate, so 
$\Omega_{H_5}$ and $\Omega_{H_6}$ are homeomorphic, with rotation by 90 degrees
taking each set to the other.  Likewise, $\Omega_{H_7}$ and $\Omega_{H_8}$
are homeomorphic. We therefore restrict our attention to $\Omega_{H_5}$ and 
$\Omega_{H_7}$. 

We begin with $\Omega_{H_5}$. 
Since all vertices with incoming and outgoing arrows have either 3 or
0 incoming arrows, and since vertices with 3 incoming arrows cannot be
symmetric under $H_5$, the tiles along the $x$ axis must alternate
between the form $\begin{matrix}\neabox \\ \seabox \end{matrix}$ and 
$\begin{matrix}\nwabox \\   \swabox \end{matrix}$. Although the pattern
along the $x$ axis is periodic, there is a hierachy from the way that
tiles along the $x$ axis join with tiles once removed from the $x$ axis to form
clusters of four tiles, which join tiles even farther away to form clusters
of 16, and so on. $\Omega_{H_5}$ is connected and is 
homeomorphic to the dyadic solenoid
$Sol_2$, so $\check H^1(\Omega_{H_5})=\Z[1/2]$. 

\begin{figure}
\includegraphics[width=2.2truein]{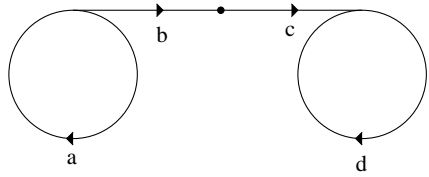}
\caption{The approximant for $\Omega_{H_7}$}
\label{H7-fig}
\end{figure}

To understand $\Omega_{H_7}$, we look at symmetric configurations of
tiles along the line $y=x$. All vertices take one of three forms:
$$ \begin{matrix} \seabox \swabox \\ \swabox \nwabox \end{matrix},\quad
\begin{matrix} \nwabox  \neabox \\ \swabox  \seabox \end{matrix}, \quad
\hbox{ and }
\begin{matrix} \seabox  \neabox \\ \neabox  \nwabox \end{matrix},$$
and the second of these patterns can occur at most once. In other words,
either all of the arrows along the line $x=y$ point northeast, or all point
southwest, or all point outwards from a special point where four 
infinite-order supertiles meet.

Reading from southwest to northeast along the line $x=y$, there
are four kinds of collared tiles that appear, which we label $a$, 
$b$, $c$ and $d$. The label
$a$ means (SW)SW(SW), while $b$ means (SW)SW(NE), $c$ means (SW)NE(NE) and 
$d$ means (NE)NE(NE). An $a$ can be followed by an $a$ or a $b$, a
$b$ is always followed by a $c$, a $c$ is always followed by a $c$, 
and a $d$ is always followed by a $d$. The Anderson-Putnam complex is then
given by the 
``eyeglasses'' graph shown
in Figure \ref{H7-fig}. 

Substitution sends edge $a$ to $aa$, edge $b$ to $ab$,
edge $c$ to $cd$ and edge $d$ to $dd$. 
The graph 
has $H^0=\Z$ and $H^1=\Z^2$, and substitution acts trivially on $H^0$ and 
by multiplication by 2 on $H^1$, so 
$\check H^0(\Omega_{H_7})=\Z$
and $\check H^1(\Omega_{H_7}) = \Z[1/2]^2$. 

One can apply a similar analysis to chair tilings in higher dimensions. 
For the 3-dimensional chair tiling, 
the relevant group is the 24-element group $G$ of symmetries of the cube, and
there are significant subgroups of order 2, 3, 4, 6, and 12. Each cyclic
subgroup $H$ gives rise to a space $\Omega_H$ with non-trivial 
$\check H^1$, while
the non-Abelian subgroups $H < G$ have $\Omega_H$ finite. 
For each non-Abelian 
$H$, the only invariant is $\check H^0(\Omega_H) = \Z^{|\Omega_H|}$.

\subsubsection{The pinwheel tilings}

The pinwheel tilings are based on a single tile, up to reflection, rotation
and translation.  It is a $1$-$2$-$\sqrt{5}$ right triangle with substitution
rule shown in Figure \ref{pin-sub}. 

\begin{figure}
\includegraphics[width=2.2truein]{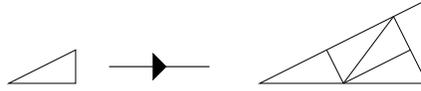}
\caption{The pinwheel substitution}
\label{pin-sub}
\end{figure}

The maximal symmetry group for any pinwheel tiling is $H_1=D_4$ 
(say, invariance under reflection about both
the $x$ and $y$ axes). There are four such tilings, all closely related. Each
is a fixed point of the square of the pinwheel substitution, with central
patches shown in Figure \ref{pin-sym1}.

\begin{figure}
\includegraphics[width=2.2truein]{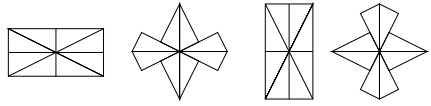}
\caption{Central patches of pinwheel tilings with dihedral symmetry}
\label{pin-sym1}
\end{figure}

There are two subgroups (up to conjugacy) of $H_1$, 
namely $H_2$ generated by 180 degree
rotation, and $H_3$ generated by reflection about the $x$ axis. 
There are six $H_2$-invariant tilings (plus rotations of the same), whose
central patches are shown in Figure \ref{pin-sym2}.
All are periodic points of the
substitution, the first four of period four and the last two of period
two. 
In other words, $\Omega_{H_2}$ consists of six disjoint circles, so
$\check H^1(\Omega_{\Z_2})=\check H^0(\Omega_{\Z_2})=\Z^6.$
These six circles are the same exceptional fibers in the fibration
$\Omega_{rot} \to \Omega_0$ that gave rise to torsion 
in $\check H^k(\Omega_{rot})$.

\begin{figure}
\includegraphics[width=2.2truein]{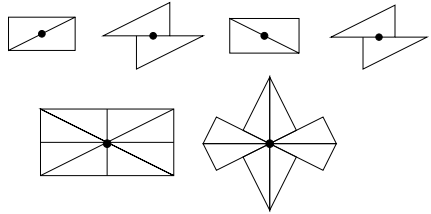}
\caption{Central patches of pinwheel tilings with rotational symmetry}
\label{pin-sym2}
\end{figure}

Finally, we consider tilings that are $H_3$-invariant. 
Since no tiles are themselves reflection-symmetric, there
must be edges along the $x$ axis, and these edges are either all 
hypotenuses or all of integer length. 

For the symmetric tilings with hypotenuses along the $x$ axis, we get a 
1-dimensional
tiling space that comes from the substitution $a \to aabba$, $b \to baabb$,
where $a$ and $b$ represent hypotenuses pointing in the two obvious
directions. This edge substitution actually comes from the square of the
pinwheel substitution, since the pinwheel substitution swaps hypotenuses
and integer legs.
A 1-dimensional Barge-Diamond calculation shows that this set of tilings 
has $\check H^1=\Z^2 \oplus \Z[1/5]$. Pinwheel substitution, 
applied only once, swaps this
space with the space of symmetric tilings involving integer edges
along the x axis, which therefore has the same cohomology.  The upshot
is that $H^1(\Omega_{H_3})=\Z^4 \oplus \Z[1/5]^2$.

\subsubsection{Asymptotic structures} 
A key difference between solenoids and spaces of 1-dimensional non-periodic 
tilings is that tilings may be forward or backwards asymptotic. 
Suppose that $\ttt_1$ and $\ttt_2$ are tilings in the same tiling
space $\Omega$, but that the restrictions of $\ttt_1$ and $\ttt_2$ to
the half-line $[0,\infty)$ are identical. Then $\ttt_1-t$ and $\ttt_2-t$
agree on a larger half-line $[-t,\infty)$, and $\lim_{t\to\infty}
d(\ttt_1-t, \ttt_2-t)=0$, where $d$ is the metric on $\Omega$. We say
that $\ttt_1$ and $\ttt_2$ are {\em forward asymptotic}. Likewise, 
two tilings can be {\em backwards asymptotic}. The orbits of $\ttt_1$ and 
$\ttt_2$ are called {\em asymptotic composants}. Every substitution tiling
space has a finite number of asymptotic composants, and the structure
of these composants is reflected in the cohomology of $\Omega$. 

For instance, in the Thue-Morse tiling space there are four periodic 
points of the substitution of the form $\ttt_1 = \ldots a.a \ldots$,
$\ttt_2 = \ldots a.b \ldots$, $\ttt_3 = \ldots b.a \ldots$ and 
$\ttt_4 = \ldots b.b \ldots$, where the central dot indicates the
location of the origin. The tilings $\ttt_1$ and $\ttt_2$ are backwards 
asymptotic, as are $\ttt_3$ and $\ttt_4$. Likewise, $\ttt_1$ and $\ttt_3$ are 
forward asymptotic, as are $\ttt_2$ and $\ttt_4$. If we imagine 
asymptotic composants to be ``joined at infinity'', then the orbits
of these four tilings form an {\em asymptotic cycle}. This asymptotic
cycle manifests itself as the closed loop on $\Gamma_{BD}$ that generated
a $\Z$ term in $\check H^1(\Omega)$. 

By studying asymptotic structures,
Barge and Diamond \cite{complete} were able to construct a complete 
homeomorphism invariant
of one-dimensional substitution tilings. Unfortunately, this invariant
is extremely difficult to compute in practice. As a practical alternative,
Barge and Smith \cite{BS} constructed an {\em augmented cohomology} of
one-dimensional substitution tilings. The precise definition involves 
the inverse limit of a variant of the Anderson-Putnam complex, but the
basic idea is to identify all forward asymptotic tilings that are periodic
points of the substitution, and separately to identify all backwards
asymptotic periodic points. The cohomology of the resulting space, while 
not a complete invariant, yields finer information than the ordinary
\v Cech cohomology. 

In higher dimensions, asymptotic structures are more subtle, since
there are (potentially) infinitely many directions to check. In 2 dimensions
(with results that generalize somewhat to still higher dimensions), Barge and
Olimb \cite{BO} examined the {\em periodic branch locus} of a substitution,
namely the set of pairs of tilings, each periodic under the substitution,
that agree on at least a half-plane. From this locus, and from translates
of these pairs along certain special directions,
they construct a larger {\em branch locus} that can have a structure similar to
that of a one-dimensional tiling space. The cohomology of the branch locus
is a homeomorphism invariant of a tiling space. 

With the chair tiling, as with a number of other examples, the branch
locus seems to be closely related to the tilings that are symmetric
under certain reflections, and the calculation of the cohomology of
the branch locus resembles the computations of $\check
H^1(\Omega_{H_5})$ and $\check H^1(\Omega_{H_7})$.  These in turn are
related to the quotient cohomology of the chair tiling space relative
to the 2-dimensional dyadic solenoid.  While it might be a
coincidence, all three computations seem to be telling the same
story!  Unfortunately, the general relation between cohomology of
branch loci, cohomology of tilings with symmetry and quotient
cohomology is not yet understood.

\subsection{Gap labeling}

For tilings of $\R^d$ with finite local complexity (with respect to
translations), there is a natural trace map from the highest
cohomology $\check H^d(\Omega)$ to $\R$. Each class $\alpha \in \check
H^d(\Omega)$ can be represented by a pattern-equivariant $d$-cochain
$i_\alpha$. Pick any bounded region $R$ of a tiling $\ttt$, let
$i_\alpha(R)$ be the sum of the values of $i_\alpha$ on all of the
tiles in $R$. If $i_\alpha$ and $i'_\alpha$ are cohomologous, then
$i_\alpha - i'_\alpha = \delta i_\beta$ for some pattern-equivariant
cochain $i_\beta$, and $i_\alpha(R) - i'_\alpha(R) = i_\beta(\partial
R)$. Define
$$Tr(\alpha) = \lim_{r \to \infty} \frac{i_\alpha(B_r)}{Vol(B_r)},$$
where $B_r$ is the ball of radius $r$ around the origin in a fixed tiling
$T$. Since $Vol(\partial B_r)/Vol(B_r) \to 0$ as $r \to \infty$, different
representatives for the class $\alpha$ yield the same limit. Likewise, 
if $\Omega$ is uniquely ergodic, then all tilings $\ttt$ yield the same limit. 

For instance, in a Fibonacci tiling where the $a$ tiles have length
$\phi=(1+\sqrt{5})/2$ and the $b$ tiles have length $1$, there are on
average $\phi$ $a$ tiles for every $b$ tile, so the indicator cochain
$i_a$ has trace $\phi/(\phi^2+1)$ and the cochain $i_b$ has trace
$1/(\phi^2+1)$. Since $i_a$ and $i_b$ generate $\check H^1$, the image
of the trace map is $(\phi^2+1)^{-1}\Z[\phi]$.

The image of the trace map is called the {\em frequency module} of the
tiling space. The frequency module is isomorphic to the {\em
  gap-labeling group}, which in K-theory is the image of a trace map
in $K^0$.  Besides being an invariant of topological conjugacies, the
gap-labeling group is used (as the name implies) to label gaps in the
spectra of Schr\"odinger operators associated with a tiling. The key
theorem is due to Bellissard (\cite{B}, see also \cite{BBG, BHZ,
  BKL}):

\begin{thm} Let $\ttt$ be a tiling in a minimal and uniquely ergodic tiling 
space $X$, and let $V: \R^d \to \R$ be a strongly pattern-equivariant function.
Consider the Schr\"odinger operator
$$ H = -\frac{\hbar^2}{2m}\Delta + V. $$
Let $E_0$ be a point that is not in the spectrum of $H$. (That is, $E_0$ lies
in a gap in the spectrum.) Then the integrated density of states up to 
energy $E_0$ is an element of the gap-labeling group of $\Omega$.
\end{thm}

Elements of the frequency module (or gap-labeling group) 
should not be viewed as pure numbers. Rather,
they have units of (Volume)${}^{-1}$, being the ratio of $i_\alpha(B_r)$ (a pure 
number) and $Vol(B_r)$. Likewise, the integrated density of states gives the
number of eigenstates of $H$ up to energy $E_0$ (a pure number) per unit
volume. 

Traces of cohomologies in all dimensions were studied in \cite{KP}, and 
are known as {\em Ruelle-Sullivan maps}. These traces give a ring homomorphism
from $\check H^*(\Omega)$ to the exterior algebra of $\R^d$.
 
\subsection{Tiling deformations}

Some properties of a tiling are consequences of the geometry of the tiles,
while others follow from the combinatorics of how tiles fit together. To
distinguish between the two, we consider different tiling spaces that have
the same combinatorics, and parametrize the possible tile shapes.

Let $X$ be a tiling space. To specify the shapes of the tiles involved, we
must indicate the displacement associated to every edge of every possible tile.
Furthermore, if two tiles share an edge, then those two edges must be described
by the same vector, and the vectors for all the edges around a tile must sum
to zero. 

In other words, the shapes of all the tiles is described by a co-closed 
vector-valued  1-cochain on a space obtained by taking one copy of each 
tile type and identifying edges that can meet. That is, a cochain on 
the Anderson-Putnam complex $\Gamma_{AP}$! Different geometric 
versions of the same combinatorial tiling space are described by different
shape covectors on the same Anderson-Putnam complex. 

\begin{thm}\cite{CS} Let $\Omega$ be a tiling space with shape 
cochain $\alpha_0$.
There is a neighborhood $U$ of $\alpha_0$ in $C^1(\Gamma_{AP}, \R^d)$ such that,
for any two co-closed shape cochains $\alpha_{1,2} \in U$, the tiling spaces 
$\Omega_{\alpha_1}$ and $\Omega_{\alpha_2}$ obtained from $\alpha_1$ and 
$\alpha_2$ are mutually locally derivable (MLD) if and only if $\alpha_1$ 
and $\alpha_2$ are cohomologous. 
\end{thm}

This theorem says that 
the first cohomology of $\Gamma_{AP}$, with values in $\R^d$,
parametrizes local deformations of the shapes and sizes of the tiles, up
to local equivalence. By considering changes in the shapes and sizes of
collared tiles and taking a limit of repeated collaring (as in G\"ahler's
construction), we obtain

\begin{thm}\cite{CS} Infinitesimal shape deformations of tiling spaces, 
modulo local equivalence, are parametrized by the vector-valued
cohomology $\check H^1(\Omega,\R^d)$. 
\end{thm}

Among all shape changes, there are some that yield topological
conjugacies. We call such a shape change a {\em shape conjugacy\/}.
Shape conjugacies correspond to a subgroup $\check H^1_{an}(\Omega, \R^d)$ of 
$\check H^1(\Omega, \R^d)$ called the {\em 
asymptotically negligible\/} classes. These classes are neatly described
in terms of pattern-equivariant functions:
\begin{thm}\cite{Kel3}
A class in $\check H^1(\Omega,\R^d)$ is asymptotically negligible if and
only if it can be represented as a strongly pattern-equivariant
vector-valued 1-form that is the differential of a weakly 
pattern-equivariant vector-valued function. 
\end{thm}

Asymptotically negligible classes don't just describe shape 
conjugacies. They essentially describe all topological conjugacies, thanks
to 
\begin{thm}\cite{KS1}
If $f: \Omega_X \to \Omega_Y$ is a topological conjugacy of tiling 
spaces, then we can write $f$ as the composition $f_1 \circ f_2$ 
of two maps, such that $f_1$ is a shape conjugacy and $f_2$ is an 
MLD equivalence. 
\end{thm}
The importance of this theorem is that it allows us to check when a 
property of a tiling (e.g. having its vertices form a Meyer set) is 
invariant under topological conjugacies. One merely has to check whether
the property is preserved by MLD maps (a local computation) and whether
it is preserved by shape conjugacies. (The Meyer property turns out to
be preserved by MLD maps but not by shape conjugacies\cite{KS1}.)

For substitution tilings, the asymptotically negligible classes
are easy to characterize:
\begin{thm}\cite{CS} Let $\Omega$ be a substitution tiling space generated
from a substitution $\sigma$. Let $\sigma^*$ denote the action of $\sigma$
on the vector space $\check H^1(\Omega, \R^d)$. The asymptotically negligible
classes are the span of the generalized eigenvectors of $\sigma^*$ with
eigenvalues strictly inside the unit circle.
\end{thm}

For example, for the Fibonacci tiling we have $\check H^1(\Omega)=\Z^2$, so
$\check H^1(\Omega, \R) = \R^2$. The substitution acts via the matrix
$\left ( \begin{smallmatrix}1&1 \cr 1&0\end{smallmatrix}\right )$, 
with eigenvalues $\lambda_1=\phi$ and $\lambda_2=1-\phi$ and eigenvectors 
$\left ( \begin{smallmatrix} \lambda_{1,2} \cr 1 \end{smallmatrix} \right)$.
Deformations proportional to the second eigenvector are asymptotically
negligible, so all deformations are locally equivalent to an overall
rescaling followed by an asymptotically negligible deformation. In particular,
any two Fibonacci tiling spaces are topologically conjugate, up to an
overall rescaling. 

Similar arguments apply to any 1-dimensional substitution tiling space
where $\sigma^*$ acts on $\check H^1$ via a Pisot matrix. The
asymptotically negligible classes have codimension 1, so all
deformations yield spaces that are topologically conjugate up to
scale. In particular, for these substitutions, suspensions of
subshifts (with all tiles having size 1) have the same qualitative
properties as self-similar tilings.

For the Penrose tiling, $\check H^1(\Omega, \R^2) = \R^5 \otimes \R^2 = 
\R^{10}$. 
The eigenvalues of $\sigma^*$ are $\phi$ and $1-\phi$, each with multiplicity
4, and $-1$ with multiplicity 2. The multiplicity 4 for the large eigenvalue
corresponds to the 4-dimensional family of linear transformations that can
be applied to $\R^2$. (For self-similar tilings of $\R^d$, the leading 
eigenvalue of $\sigma^*$ always has multiplicity $d^2$.) 
Meanwhile, the two deformations with eigenvalue -1 break the
180-degree rotational symmetry of the tiling space. Thus, any combinatorial
Penrose tiling space that maintains 180-degree rotational symmetry must be 
topologically conjugate to a linear combination applied to the ``standard''
Penrose tiling space. 

For cut-and-project tilings, the asymptotically negligible classes depend on
the shape of the ``window''. When the window isn't too complicated, 
there is an explicit description of these classes. This theorem applies even when 
$\check H^1(\Omega, \R^d)$ is infinite-dimensional. 
\begin{thm}\cite{KS2}If the window of a cut-and-project scheme of codimension $n$ is a polytope, or a
finite union of polytopes, then $\dim(\check H^1_{an}(\Omega, \R^d))=nd$. The elements of 
$\check H^1_{an}(\Omega, \R^d)$ correspond to projections from $\R^{n+d}$ to $\R^d$, and all shape conjugacies
amount  to simply changing the projection by which points in the acceptance strip are sent to $\R^d$.
\end{thm}

Besides the cohomology of strongly PE functions and forms, we can consider the {\em weak PE} cohomology of 
weakly PE functions and forms, and the {\em mixed} cohomology \cite{Kel3}. 
We call a strongly PE form {\em weakly exact\/} if it can be written as $d$ of a weakly PE $(k-1)$-form. 
The $k$th mixed cohomology $H^k_{PE,m}(\ttt)$ of a tiling $\ttt$
is the quotient of the closed strongly PE $k$-forms by the weakly exact $k$-forms. This should not be viewed 
as a subgroup of 
$H^k_{PE}(\ttt) \equiv \check H^k(\Omega, \R)$. Rather, it is a quotient of $H^k_{PE}(\ttt)$ by those classes that
can be representedy by weakly exact forms. This can be identified with a quotient of $\check H^k(\Omega, \R)$. In dimension 1,  
$$H^1_{PE,m}(\ttt) \equiv \check H^1(\Omega,\R) / H^1_{an}(\Omega, \R).$$
Since $H^1_{an}(\Omega, \R^d)$ parametrizes shape conjugacies, this means that $H^1_{PE,m}(\ttt)$ parametrized 
deformations of a tiling space $\Omega_\ttt$
{\em up to topological conjugacy\/} rather than up to MLD equivalence \cite{Kel3}.

\subsection{Exact regularity}

A measure on a 
tiling space is equivalent to specifying the frequencies
of all possible patches.  Specifically, let $P$ be a patch in a specific
location (say, centered at the origin). Let $U$ be an open set in $\R^d$.
Let $\Omega_{P,U}$ be the set of all tilings $\ttt$ such that, for some $x\in U$,
$\ttt-x$ contains the patch $P$. In other words, $\ttt$ must contain the patch
$P$ at location $x$. As long as $U$ is chosen small enough, there is at 
most one $x \in U$ that works. For any tiling $\ttt$, let $freq_T(P)$ be the
number of occurrences of $P$, per unit are, in $\ttt$. That is, restrict
$\ttt$ to a large ball,
divide by the volume of the ball, and take a limit as the radius goes to
infinity. The ergodic theorem says that this limit exists for $\mu$-almost
every $\ttt$, with $freq_\ttt(P) = \mu(\Omega_{P,U})/Vol(U)$. If the 
tiling space is
uniquely ergodic, then this statement applies to every $\ttt$, not just to
almost every $\ttt$. 

There are two natural questions. First, what are the possible values of 
$freq_\ttt(P)$? Second, as we consider larger and larger balls, how quickly does
the number of occurrences of $P$ per unit area approach $freq_\ttt(P)$? Both
questions have cohomological answers. 

\begin{thm} For each patch $P$ and each sufficiently small open subset $U$
of $\R^d$, $\mu(\Omega_{P,U})/Vol(U)$ takes values in the frequency module of 
$X$. 
\end{thm}

\begin{proof} Let $i_P$ be a pattern-equivariant $d$-cochain that equals
1 on one of the tiles of $P$ and is zero on all other tiles. Being
of the top dimension, $i_P$ is co-closed, and so represents a cohomology
class. For any region
$R$, $i_P(R)$ is just the number of occurrences of $P$ in $R$. The limiting
number per unit area $freq_\ttt(P)$ is then the trace of the class of $i_P$.
\end{proof}

\begin{thm} \cite{ERP2}
Suppose that $\check H^d(\Omega,\Q) = \Q^k$ for some integer $k$. 
Then there exist patches $P_1$, \ldots, $P_k$ 
with the following property: for any other patch $P$, there exist rational
numbers $c_1(P)$, \ldots, $c_k(P)$ such that, 
for any region $R$ in any tiling $\ttt \in X$, the number 
of appearances of $P$ in $R$ equals $\sum_{i=1}^k c_i(P)n_i(R) + e(P, R)$, 
where $n_i(R)$ is the number of
appearances of $P_i$ in $R$, and $e(P, R)$ 
is an error term computable from the patterns that
appear on the boundary of $R$. In particular, the magnitude of $e(P, R)$ 
is bounded by a constant (that may depend on $P$) 
times the measure of the boundary of $R$. Furthermore, if $\check H^d(\Omega)
= \Z^k$ is finitely generated over the integers, then we can pick the 
coefficients $c_i$ to be integers. 
\end{thm}

\begin{cor}
If the patches $P_1$, \ldots, $P_k$ have well-defined frequencies, then
$\Omega$ is uniquely ergodic and there exist uniform bounds for the convergence
of all patch frequencies to their ergodic averages. 
If the regions $R$ are chosen to be balls, whose radii we denote $r$, then 
the number of $P$'s per unit area approaches $\sum c_i freq(P_i)$ at least
as fast as one the frequency of one of the $P_i$'s approaches $freq(P_i)$, 
or as fast as $r^{-1}$, whichever is slower.
\end{cor}
Note that this theorem and its corollary 
apply to all tiling spaces, and not just to 
substitution tiling spaces. 

\begin{proof} Using the isomorphism between \v Cech and pattern-equivariant
cohomology, pick patches $P_1$, \ldots, $P_k$ such that the cohomology
classes of $i_{P_i}$ are linearly independent. These classes then form a 
basis for $H_{PE}^d(X) = \Q^k$, and we can write $[i_P] = \sum c_i
[i_{P_i}]$, where $[\alpha]$ denotes the cohomology class of the cochain
$\alpha$. This means that there is a $(d-1)$-cochain $\beta$ such that
$i_P = \sum_i c_i i_{P_i} + \delta \beta$. Then
\begin{equation} \hbox{number of $P$ in $R$} = i_P(R) = 
\sum c_i i_{P_i}(R) + \delta \beta(R) = \sum c_i n_i(R) + \beta(\partial R).
\end{equation}
Since $\beta$ is pattern-equivariant there is a maximum value that it takes
on any $(d-1)$-cell, so the error term
$\beta(\partial R)$ is bounded by a constant times the area of the boundary
of $R$. Dividing by the volume of $R$, the deviation of the left-hand-side
from $freq(P)$ is bounded by the deviation of $n_i(R)/Vol(R)$ from
$freq(P_i)$ or by $|\partial R|/Vol(R) \sim r^{-1}$.       

If $\check H^d(\Omega) = \Z^k$, then the same argument applies with the patches
chosen such that $[i_{P_i}]$
are generators of $\Z^k$ and with integral coefficients $c_i$. 
\end{proof}

When $\Omega$ is a 1-dimensional tiling space, it is possible to pick $R$
such that $\partial R$ is homologically trivial. Let $\beta$ be 
pattern-equivariant with radius $r_0$, and let $W$ be a word of length
at least $2r_0$. Pick $R$ to be an interval that starts in the middle
of one occurrence of $W$ and ends in the corresponding spot of another
occurrence. Then $\delta \beta(R) = \beta(\partial R)$ vanishes, and 
the number of $P$'s in $R$ is {\em exactly} $\sum_i c_i(P) n_i(R)$. 
This is called {\em exact regularity} \cite{ERP1, ERP2}. 

\subsection{Invariant measures and homology}

Exact regularity is dual to an earlier description \cite{BG} 
of invariant measures
in terms of the real-valued homology $H_d(\Gamma^n,\R)$ of the approximants. 
Recall that measures do not pull back, but instead {\em push forward}
like homology classes: Given a measure $\mu$ on a space
$X$ and a continuous map $f: X \to Y$, there is a measure $f_*\mu$ on $Y$.
For any measurable set $S \subset Y$, $f_*\mu(S) := \mu(f^{-1}(S))$. Thus, a
measure $\mu$ on a tiling space gives rise to a sequence of measures 
$\mu_n$ on the approximants $\Gamma^n$, with $(\rho_n)_* \mu_n = \mu_{n-1}$.

Integration gives a pairing between indicator $d$-cochains and measures.
$\langle \mu, i_P\rangle = freq(P)$. This extends to a pairing
between measures and cohomology, both for the tiling space and for each
approximant. By the universal coefficients theorem, the dual space to
the top cohomology group $H^d(\Gamma^n , \R)$ is 
the top homology group $H_d(\Gamma^n, \R)$, so we can view $\mu_n$
as living in $H_d(\Gamma^n ,\R)$, and $\mu$ as living in the inverse limit
space $\ilim (H_d(\Gamma^n,\R), (\rho_n)_*).$  (The identification of 
$\mu_d$ as an element of $H_d(\Gamma^n, \R)$ can also be seen more directly.
A measure can be viewed as a chain satisfying certain ``switching rules'', or
``Kirchoff-like laws''. These rules are equivalent to saying that the boundary
operator applied to $\mu_d$ is zero, i.e. that $\mu_d$ defines a homology 
class.) 

The measure of any cylinder set is non-negative, so each $\mu_d$ must lie
in the positive cone of $H_d(\Gamma^n, \R)$. Not only is $\mu$ 
constrained to lie in the inverse limit of the top homologies of the
approximants, but $\mu$ must lie in the inverse limit of the positive cones. 
All of the invariant measures on a tiling space can be determined from
the transition matrices $(\rho_n)_*$,
and in particular we can tell whether the tiling space is uniquely ergodic.


\begin{thebibliography}{ieeetr}

\bibitem[AFHI]{Marseille} P.~Arnoux, M.~Furukado, E.~Harriss and S.~Ito,
Algebraic numbers, free group automorphism and substitution on the plane, 
{\em Trans. Amer. Math. Soc.} {\bf 363} (2011) 4651–-4699.

\bibitem[AP]{AP} J. E. Anderson and I. F. Putnam, Topological invariants for
substitution tilings and their associated $C^{\ast }-$algebras, \emph{%
Ergodic Theory  Dynam. Systems }\textbf{18} (1998), 509--537.

\bibitem[Bel]{B} J.~Bellissard, The gap-labeling theorems for
Schr\"odinger operators, in ``From number theory to physics'', pp. 538--630,
Les Houches March 1989. J.M.~Luck, P.~Moussa and M.~Waldschmidt eds.

\bibitem[BBG]{BBG} 
J.~Bellissard,  R.~Benedetti, and J.-M.~Gambaudo,
Spaces of tilings, finite telescopic approximations and gap-labelling,
{\em Comm. Math. Phys.} {\bf 261} (2006) 1--41.


\bibitem[BBJS]{ERP1} M.~Barge, H.~Bruin, L.~Jones and L.~Sadun,
Homological Pisot Substitutions and Exact Regularity, 
{\em Israel Journal of Mathematics} {\bf 188} (2012), 281-300.

\bibitem[BD1]{complete} M.~Barge and B.~Diamond, 
A complete invariant for the topology 
of one-dimensional substitution tiling spaces,
{\em Ergodic Theory Dynam. Systems} \textbf{21} (2001) 1333--1358.


\bibitem[BD2]{BD} M.~Barge and  B.~Diamond, Cohomology in one-dimensional
substitution tiling spaces, {\em Proc. Amer. Math. Soc} {\bf 138} (2008)
2183--2191.

\bibitem[BDHS]{BDHS} M.~Barge, B.~Diamond, J.~Hunton and L.~Sadun, 
Cohomology of Substitution Tiling Spaces, {\em 
Ergodic Theory and Dynamical Systems} {\bf 30} (2010) 1607--1627.

\bibitem[BG]{BG} 
R.~Benedetti, J.-M.~Gambaudo, 
On the dynamics of $G$-solenoids: applications to Delone sets,
{\em Ergodic Theory and Dynamical Systems}  {\bf 23} (2003) 673--691.





\bibitem[BHZ]{BHZ} J.~Bellissard, D.~Herrmann, and M.~Zarrouati,
Hulls of aperiodic solids and gap labeling theorems, in ``Directions
in mathematical quasicrystals'',  CRM monograph series {\bf 13}
(2000) 207--259.  M.B.~Baake and R.V.~Moody eds. 


\bibitem[BK]{HK} H.~Boulezaoud and J.~Kellendonk,
Comparing different versions of tiling cohomology,
{\em Topology and its applications} {\bf 157} (2010) 2225--2239.
 
\bibitem[BKL]{BKL} J.~Bellissard, J.~Kellendonk and A.~Legrand,
Gap labeling for three dimensional aperiodic solids, {\em C.R.
Acad. Sci} {\bf 332}, S\'erie I (2001) 521--525.

\bibitem[BO]{BO} M.~Barge and C.~Olimb, Asymptotic structure in 
substitution tiling spaces, {Ergodic Theory and Dynamical Systems} {\bf 34}
(2014) 55--94.

\bibitem[BSa]{quotient} 
M.~Barge and L.~Sadun, Quotient Cohomology for Tiling Spaces, 
{\em New York Journal of Mathematics} {\bf 17} (2011) 579--599.

\bibitem[BSm]{BS}
M.~Barge and M.~Smith, Augmented dimension groups and ordered cohomlogy,
{\em Ergodic Theory and Dynamical Systems} {\bf 29} (2009) 1--35.

\bibitem[BT]{BT} R.~Bott and L.~Tu, ``Differential Forms in Algebraic
Topology'', Springer-Verlag, 1982.
  

\bibitem[CS]{CS} A.~Clark and L.~Sadun, 
When Shape Matters: Deformations of Tiling Spaces, 
{\em Ergodic Theory and Dynamical Systems} {\bf 26} (2006) 69--86.

\bibitem[Da]{Danzer} Inflation species of planar tilings which are not of
locally finite complexity. {\em Proc. Steklov Inst. Math.} {\bf 230} (2002),
118--126. 




\bibitem[FHK]{FHK} A.~Forrest, J.~Hunton and J.~Kellendonk, 
Topological Invariants for Projection Method Patterns, Memoirs of the
Amer. Math. Soc. {\bf 758}, (2002).

\bibitem[F]{F} N.P.~Frank, A primer on substitution tilings of the Euclidean plane,
{\em Expositiones Mathematicae} {\bf 26} (2008), 295--326.


\bibitem[FR]{Frank} N.P.~Frank and E.A.~Robinson, Jr., 
Generalized $\beta$-expansions, substitution tilings, 
and local finiteness,
{\em Trans. Amer. Math. Soc.} {\bf 360} (2008) 1163--1177.



\bibitem[FS]{FS} N.P.~Frank and L.~Sadun,
Topology of (Some) Tiling Spaces without Finite Local Complexity,
{\em Discrete and Continuous Dynamical Systems} {\bf 23} (2009) 847-865.


\bibitem[Gah]{Gaehler} F. G\"ahler, talk given at the conference
``Aperiodic Order, Dynamical Systems, Operator Algebras and Topology'', 2002,
slides available at {\tt www.pims.math.ca/science/2002/adot/lectnotes/Gaehler}


\bibitem[GS]{GS} C.~Goodman-Strauss, 
Matching Rules and Substitution Tilings, {\em Annals of Mathematics}, 
{\bf 147} (1998), 181--223.

\bibitem[Hat]{Hatcher} A.~Hatcher ``Algebraic Topology'', (2002) Cambridge
University Press. 

\bibitem[Kal]{Kalugin} P.~Kalugin, Cohomology of quasiperiodic patterns and
matching rules, {\em J.~Phys.~A.} {\bf 38} (2005) 3115--3132.


\bibitem[Kel1]{Kel-force-border} 
Noncommutative geometry of tilings and gap-labelling,
{\em Rev. Math. Phys.} {\bf 7} (1995), 1133--1180.

\bibitem[Kel2]{Kel2} J.~Kellendonk, 
Pattern-equivariant functions and cohomology, 
\emph{J.~Phys.~A} {\bf 36} (2003) 5765--5772. 

\bibitem[Kel3]{Kel3} J.~Kellendonk, 
Pattern-equivariant functions, deformations and equivalence 
of tiling spaces, 
{\em Ergodic Theory \& Dynamical Systems} {\bf 28} (2008) 1153--1176. 

\bibitem[Ken]{Kenyon} R.~Kenyon, Self-Replicating Tilings, in ``Symbolic
dynamics and its applications'', AMS Contemp. Math. Series 135, P.~Walters
ed., (1992) 239--263.


\bibitem[KP]{KP} J.~Kellendonk and I.~Putnam,
The Ruelle-Sullivan map for $R^n$ actions, 
{\em Math. Ann.}  {\bf 334} (2006), 693--711.

\bibitem[KS1]{KS1} J.~Kellendonk and L.~Sadun, 
Meyer sets, topological eigenvalues, and Cantor fiber bundles Journal of the London Math Society (2013), online 10.1112/jlms/jdt062 

\bibitem[KS2]{KS2} J.~Kellendonk and L.~Sadun, Conjugacies of 
Model Sets, preprint 2014. 


\bibitem[Mos]{Mosse} B. Mossé: 
Puissances de mots et reconnaissabilité des point fixes d'une substitution. 
{\em Theor. Comput. Sci.} {\bf 99} (1992) 327--334.

\bibitem[Moz]{Mozes} S. Mozes, Tilings, substitution systems and
dynamical systems generated by them, \emph{J. Analyse Math.} {\bf 53}
(1989), 139--186.

\bibitem[ORS]{ORS} N.~Ormes C.~Radin and L.~Sadun, 
A Homeomorphism Invariant for Substitution Tiling Spaces, 
{\em Geometriae Dedicata} {\bf 90} (2002), 153--182. 





\bibitem[Rad]{pinwheel} C.~Radin, The pinwheel tilings of the plane,
\emph{Annals of Math.} \textbf{139} (1994) 661--702.






\bibitem[Sa1]{inverse} L.~Sadun, Tiling spaces are inverse limits,
{\em J.~Math.~Phys.} {\bf 44} (2003) 5410--5414.

\bibitem[Sa2]{integer} L.~Sadun, Pattern Equivariant Cohomology with
Integer Coefficients, {\em Ergodic Theory Dynan. Sys.} {\bf 27} (2007)
1991--1998. 

\bibitem[Sa3]{book} L.~Sadun, ``Topology of Tiling Spaces'', University
Lecture Series, Vol. 46, American Mathematical Society, Providence, RI 2008. 

\bibitem[Sa4]{ERP2} L.~Sadun, Exact Regularity and the Cohomology of 
Tiling Spaces, {\em Ergodic Theory and Dynamical Systems} 
{\bf 31} (2011) 1819--1834. 

\bibitem[SB]{SB} J.~Savinien and J.~Bellissard, A spectral sequence for the 
K-theory of tiling spaces, {\em Ergod. Th. \& Dynam. Sys.} {\bf 29} (2009) 
997--1031



\bibitem[Sol]{Solomyak} B.~Solomyak, 
Nonperiodicity implies unique decomposition for self-similar translationally
finite tilings.
{\em Disc. \& Comp. Geom.} {\bf 20} (1998), 265--279. 




\end{thebibliography}
\end{document}